\documentclass[12pt]{article}

\usepackage{amsmath,latexsym,amsfonts,amssymb,mathrsfs,amsthm}

\usepackage{verbatim}

\usepackage{euscript}

\usepackage{algpseudocode,algorithm} 
\floatname{algorithm}{Pseudocode} 

\usepackage[usenames]{color}
\definecolor{darkred}{RGB}{139,0,0}
\definecolor{darkgreen}{RGB}{0,100,0}
\definecolor{darkmagenta}{RGB}{139,0,139}
\definecolor{darkpurple}{RGB}{110,0,180}
\definecolor{darkblue}{RGB}{40,0,200}
\definecolor{darkorange}{RGB}{255,140,0}

\usepackage[colorlinks=true,linkcolor=black,citecolor=black]{hyperref}

\newcommand{\jd}[1]{{\color{darkred}{#1}}}
\newcommand{\beqn}{\begin{equation}}
\newcommand{\eeqn}{\end{equation}}

\newcommand{\fa}{{\mathfrak a}}

\newcommand{\ff}{{\mathfrak f}}

\def\N{\mathbb{N}}
\def\IC{\mathbb{C}}

\def\IN{\mathbb{N}}

\def\IR{\mathbb{R}}

\newcommand{\C}{\mathbb{C}}

\newcommand{\LinFun}{\mathscr{O}}  

\newcommand{\Gc}{\mathscr{P}} 

\newcommand{\calD}{\mathcal{D}}
\newcommand{\cE}{\mathcal{E}}
\newcommand{\cF}{\mathcal{F}}

\newcommand{\cR}{\mathcal{R}}
\newcommand{\cT}{\mathcal{T}}

\newcommand{\cO}{\mathcal{O}}

\newcommand{\cX}{\mathcal{X}}
\newcommand{\bcX}{\mathcal{X}}
\newcommand{\bcY}{\mathcal{Y}}
\newcommand{\cY}{\mathcal{Y}}

\newcommand{\eps}{\varepsilon}

\newcommand{\satop}[2]{\stackrel{\scriptstyle{#1}}{\scriptstyle{#2}}}

\newcommand{\bsalpha}{{\boldsymbol{\alpha}}}
\newcommand{\bsbeta}{\boldsymbol{\beta}}
\newcommand{\bsgamma}{{\boldsymbol{\gamma}}}
\newcommand{\bsrho}{{\boldsymbol{\rho}}}
\newcommand{\bstau}{\boldsymbol{\tau}}

\newcommand{\bsnu}{{\boldsymbol{\nu}}}
\newcommand{\bsell}{{\boldsymbol{\ell}}}
\newcommand{\bsb}{{\boldsymbol{b}}}
\newcommand{\bsq}{{\boldsymbol{q}}}
\newcommand{\bsk}{{\boldsymbol{k}}}

\newcommand{\bsd}{\boldsymbol{d}}

\newcommand{\bsy}{{\boldsymbol{y}}}
\newcommand{\bsz}{{\boldsymbol{z}}}

\newcommand{\bsS}{\boldsymbol{S}}
\newcommand{\bsU}{\boldsymbol{U}}
\newcommand{\bsV}{\boldsymbol{V}}
\newcommand{\bsW}{\boldsymbol{W}}
\newcommand{\bsX}{\boldsymbol{X}}
\newcommand{\bsY}{\boldsymbol{Y}}

\newcommand{\rd}{\mathrm{d}}

\newcommand{\bbC}{\mathbb{C}}

\newcommand{\bbZ}{\mathbb{Z}}
\newcommand{\bbN}{\mathbb{N}}

\newcommand{\calO}{\mathcal{O}}
\newcommand{\calW}{\mathcal{W}}
\newcommand{\cL}{\mathcal{L}}
\newcommand{\setu}{\mathrm{\mathfrak{u}}}
\newcommand{\setv}{\mathrm{\mathfrak{v}}}
\newcommand{\setw}{\mathrm{\mathfrak{w}}}

\newcommand{\KL}{Karhunen-Lo\`eve }

\newcommand{\setD}{{\mathfrak{D}}}
\newcommand{\mask}[1]{{}}
\newcommand{\bszero}{{\boldsymbol{0}}}

\numberwithin{equation}{section}

\theoremstyle{plain}
  \newtheorem{theorem}{Theorem}[section]
  
  \newtheorem{lemma}{Lemma}[section]
  \newtheorem{corollary}{Corollary}[section]
  \newtheorem{proposition}{Proposition}[section]
  \newtheorem{assumption}{Assumption}
\theoremstyle{definition}
   \newtheorem{definition}{Definition}[section]
   \newtheorem{remark}{Remark}[section]

\newcommand{\bsOmega}{{\boldsymbol{\Omega}}}

\newcommand{\bsone}{{\boldsymbol{1}}}

\newcommand{\Z}{\mathbb{Z}}

\newcommand{\calS}{\mathcal S}

\newcommand{\be}{\begin{equation}}
\newcommand{\ee}{\end{equation}}
\newcommand{\ba}{\begin{array}}
\newcommand{\ea}{\end{array}}
\newcommand{\beas}{\begin{eqnarray*}}
\newcommand{\eeas}{\end{eqnarray*}}
\newcommand{\bea}{\begin{eqnarray}}
\newcommand{\eea}{\end{eqnarray}}

\DeclareMathOperator{\supp}{supp}

\title{Higher order Quasi-Monte Carlo integration for 
       \\
       holomorphic, parametric operator equations\thanks{
The research of JD was supported under Australian Research Council's Discovery Projects funding scheme (project number  DP150101770), the research of QLG was supported under Australian Research Council's Discovery Projects funding scheme (project number  DP120101816) and the 
work of CS was supported in part by the European Research Council (ERC) under grant 
AdG 247277, and by the Swiss National Science Foundation (SNF) under grants
SNF 200021-159940 and SNF 200021-149819.
}
\author{
Josef Dick, Quoc T. Le Gia and Christoph Schwab}
}

\date{\today}

\begin{document}
\maketitle
\begin{abstract}
We analyze the convergence of higher order
Quasi-Monte Carlo (QMC) quadratures
of solution-functionals 
to countably-parametric, nonlinear operator equations
with distributed uncertain parameters taking values in a
separable Banach space $X$ admitting an unconditional Schauder basis.

Such equations arise in numerical uncertainty quantification
with random field inputs. 
Unconditional bases of $X$ render the random inputs
and the solutions of the forward problem countably parametric,
deterministic.
We show that these parametric solutions belong 
to a class of weighted Bochner 
spaces of functions of countably many variables,
with a particular structure of the QMC quadrature weights:
up to a (problem-dependent, and possibly large) 
finite dimension, product weights can be used, 
and beyond this dimension,
weighted spaces with so-called SPOD weights recently
introduced in [F.Y.~Kuo, Ch.~Schwab, I.H.~Sloan, 
Quasi-Monte Carlo finite element methods for a class 
of elliptic partial differential equations 
with random coefficients. SIAM J. Numer. Anal., 50, 3351--3374, 2012.] 
can be used to describe the solution regularity.
The regularity results in the present paper 
extend those in [J. Dick, F.Y.~Kuo, Q.T.~Le Gia, D.~Nuyens, Ch.~Schwab, 
Higher order QMC (Petrov-)Galerkin discretization for parametric operator equations. SIAM J. Numer. Anal., 52, 2676 -- 2702, 2014.] 
established for affine parametric, linear operator families; 
they imply, in particular, efficient constructions of 
(sequences of) QMC quadrature methods there, 
which are applicable to these problem classes.
We present a hybridized version of the fast
component-by-component (CBC for short) construction
of a certain type of higher order digital net.
We prove that this construction 
exploits the product nature of the QMC weights 
with linear scaling with respect to the integration dimension 
up to a possibly large, problem dependent finite dimension,
and the SPOD structure of the weights with quadratic scaling
with respect to the weights beyond this dimension.
\end{abstract}

\section{Introduction}
\label{sec:Intro}
The numerical computation of statistical quantities for
solutions of operator equations 
which depend on ``uncertain input parameters''
is a key task 
in uncertainty quantification in engineering and in
the sciences. 
We consider here the case when the uncertain input
quantities are random variables taking values
in subsets of an infinite-dimensional, 
separable Banach space $X$.
The system's {\em responses} 
to such random inputs are, in turn,
random variables taking values in a state space $\bcX$.
One is interested in 
{\em statistical moments} of these random responses, 
such as the mean response and (co)variance. These, and other
{\em quantities of interest (QoI)} are then expressed as 
mathematical expectations over all realizations of 
the uncertain input $u\in X$. 

The numerical approximation of such QoI's in 
these problems involves two basic steps: 
i) approximate (numerical) solution of the operator equation, 
and 
ii) approximate evaluation of the mathematical expectation by 
    dimension-truncation and some form of dimension-robust 
        numerical integration, i.e. an integration method that is
     free from the curse of dimensionality 
     under certain assumptions on the integrand.
In the present paper, we outline a strategy towards these two 
aims, which is based on
i) a \emph{(Petrov-)Galerkin discretization} of the 
parametric, nonlinear operator equation 
and 
on ii) higher order \emph{QMC} integration.
It is motivated in part by \cite{KSS12}, where QMC
integration using a family of \emph{randomly shifted lattice rules} 
was combined with a Finite Element discretization for a model 
linear, parametric diffusion equation, 
and in part by \cite{ScMCQMC12}, 
where the methodology of \cite{KSS12} was extended to 
problems described by an abstract family of 
{\em linear and affine-parametric} operator equations. 

In contrast to \cite{KSS12,ScMCQMC12}, 
we propose and analyze the convergence of
\emph{deterministic}, so-called \emph{``interlaced polynomial lattice rules''}
for the numerical evaluation of infinite-dimensional
integrals for integrand functions obtained 
from Petrov-Galerkin (PG) discretization of parametric operator equations
with random input. We allow in particular \emph{distributed uncertain input data}
taking values in a separable Banach space $X$ 
which entails, upon parametrization with an unconditional basis,
infinitely many parameters. 
High order QMC quadratures are proved to provide 
dimension-independent convergence rate 
beyond order one for smooth integrands (cf. \cite{D08,D09}); 
convergence order one was the limitation in \cite{KSS12,KSS11,ScMCQMC12}.

In the present paper, we generalize these works and
prove that \emph{sparsity} of the uncertain input implies
higher order, dimension-independent convergence rates for the 
QMC evaluation of expectations of QoI's 
(under a probability measure on the space $X$ of uncertain inputs) 
for a class of nonlinear, parametric operator equations.

The outline of this paper is as follows:
in Section \ref{sec:HolOpEq}, 
we introduce a class of nonlinear,
holomorphic-parametric operator equations with sufficient
conditions on the nonlinear operators and on the uncertainty for the 
problems to be well-posed, from \cite{CCS2,GR90,PR}. 
We require that these conditions hold
uniformly on a set $\tilde{X} \subset X$ 
of admissible uncertainties.
We give a parametrization of the uncertain inputs which reduce the 
problem to a parametric, deterministic 
integration problem which depends on a possibly countable number of parameters 
$y_j\in [-1,1]$.
We review the theory of (Petrov-)Galerkin discretizations
of these equations, and develop discretization error estimates.
In Section~\ref{sec:Hol} we review the notion
of {\em holomorphy} of the integrand functions in these problems, 
from \cite{CCS2}, whereas in Section~\ref{sec:anadepsol}, 
we present the first principal result of the present paper on analyticity and parametric
regularity of the parametric integrand functions. 
Section~\ref{sec:QMC} presents the convergence theory 
for higher order QMC quadratures, based on \cite{DiPi07,DiPi10},
and, for the parametric integrands appearing here, on \cite{DKGNS13}.

In Theorem \ref{thm:Combin}, we prove an error bound with 
dimension-independent constants and convergence rates which
accounts for all sources of discretization error in the presently
proposed class of algorithms: i) dimension truncation in
the parametrized uncertain input $u\in X$, ii) 
(single-level) Petrov-Galerkin discretization of parametric operator
equation and iii) Higher order quasi-Monte Carlo quadrature
approximations of integrals of the dimensionally truncated,
parametric quantities of interest.

Based on the results in Section~\ref{sec:anadepsol}, 
the second principal result of this paper in 
Section \ref{sec:HybrCBC} pertains to 
new variants of the fast component-by-component
CBC constructions of generating vectors, which are developed
based on \cite{DKGNS13,DiGo12,Go13} and which are tailored to
the `hybrid' nature of the QMC weights, 
with possibly more favorable complexity estimates for the CBC construction.
\section{Holomorphic parametric operator equations}
\label{sec:HolOpEq}
%
We present a class of operator equations
which depend on an uncertain, ``distributed parameter'',
being an element $u$ in a real, separable Banach space $X$.
For a given, known forcing term $f\in \cY'$, and 
any instance of $u$ in (a subset of) $X$, 
the operator equation will admit a unique solution
(also referred to as ``response'')
$q\in \cX$; here, $\cX$ and $\cY$ are assumed to
be real, separable and reflexive Banach spaces and $\cY'$ is the dual space of $\cY$.
In this section, we present a mathematical setting 
which accommodates this kind of problem and
introduce conditions which ensure the (Lipschitz)
continuous dependence of the response $q\in\cX$ on the 
uncertain input $u\in X$. 
Assuming $X$ to be separable
and to admit an unconditional Schauder basis
$\Psi = \{ \psi_j \}_{j\geq 1}$, 
with an eye towards QMC algorithms,
we reformulate the operator equation with
distributed uncertain input as infinite-dimensional, 
parametric operator equation where the uncertain input $u$
is replaced by the sequence $\bsy$ of its coefficients $y_j$
with respect to the basis $\Psi$. 
We then provide
error bounds of the response subject to $s$-term 
truncations of the basis representation of $u$ in terms of 
the basis $\Psi$.
We also provide a general framework, from \cite{PR},
for {\em Petrov-Galerkin approximation} of the responses
$q\in \cX$, and bound the combined error due to 
dimension-truncation and Petrov-Galerkin approximation.
The derivative bounds of multivariate integrand functions
necessary for QMC convergence theory will be based on
\emph{analytic continuation} with respect to the
integration variable into the complex domain.
To this end, we review in Section \ref{sec:Hol} 
a holomorphy result from \cite{CCS2} for
the parameter
dependence of the uncertainty-to-response 
map $X\ni u\rightarrow q\in \cX$; to this end, we extend 
in Section \ref{sec:Hol} the
Banach spaces $X$, $\cX$ and $\cY$ to the coefficient field $\IC$.
\subsection{Nonlinear operator equations with uncertain input data}
\label{sec:OpEqUncInp}
For a distributed, uncertain parameter $u\in X$,
we consider a possibly nonlinear operator equation with input $u$ 
which is defined by a ``residual'' operator 
$\cR: X \times \cX \to \cY'$, where $\cR(u;q)$ 
acts, for given $u$, on $q\in \bcX$.
We assume a known {\em ``nominal parameter instance''}
      $\langle u \rangle \in X$
(such as, for example, the expectation of an $X$-valued random field $u$),
and consider, for $u \in {\mathscr B}_X(\langle u \rangle;R)$, 
an open ball of radius $R>0$ in $X$ centered at
$\langle u \rangle\in X$, 
the following class of smooth, parametric, nonlinear 
operator equations,
%
%
%
\be\label{eq:NonOpEqn}
\mbox{given}\;\;u\in {\mathscr B}_X(\langle u \rangle;R)\;,
\;\mbox{find} \; q\in \bcX \quad \mbox{s.t.} \quad 
_{\bcY'}\langle \cR(u;q) , v \rangle_\bcY = 0  \quad \forall v\in \bcY
\;,
\ee
where $_{\bcY'}\langle\cdot,\cdot\rangle_\bcY$ 
denotes the $\bcY'\times \bcY$-duality pairing.

Given $u\in {\mathscr B}_X(\langle u \rangle;R)$,
we call a solution $q_0$ of \eqref{eq:NonOpEqn} {\em regular
at $u$} iff $\cR(u;\cdot)$ is Fr\'{e}chet differentiable with respect to $q$
and the differential $D_q\cR(u;q_0)\in \cL(\bcX,\bcY')$ is an
isomorphism \jd{(here $\cL(\bcX,\bcY')$ denotes the set of all bounded linear functionals from $\bcX$ to $\bcY'$)}.
We impose further structural conditions on $\cR$:
for every admissible $u\in \tilde{X} \subseteq X$, 
we assume given a parametric forcing functional 
$F(u)\in \bcY'$, 
and a parametric, nonlinear operator $A(u;q): X\times \bcX \rightarrow \bcY'$,
so that \eqref{eq:NonOpEqn} is equivalent to finding,
for every $u \in  {\mathscr B}_X(\langle u \rangle;R)$,
$q(u)\in \bcX$ which satisfies the residual equation
\begin{equation}\label{eq:main}
\cR(u;q) = A(u;q) - F(u) = 0 \quad \mbox{in}\quad \bcY'\;.
\end{equation}

Problems of the form \eqref{eq:main} (i.e.,
with separate expressions $A$ and $F$ for the 
uncertain system resp. its forcing) arise in 
a number of applications; in the  form \eqref{eq:main}, 
the equation $A(u;q) = F(u)$ is obviously a special case
of \eqref{eq:NonOpEqn}. 

In the remainder of this article,
we develop sufficient conditions for unique solvability
for the parametric weak residual equation \eqref{eq:main}.
Sufficient conditions on $\cR$ for unique solvability of
\eqref{eq:main} straightforwardly imply analogous conditions
on $A$ and on $F$ in \eqref{eq:NonOpEqn} which we will not
detail in each case.

For the well-posedness of operator equations
involving $\cR(u;q)$ 
we assume the map $\cR(u;\cdot):\bcX \to \bcY'$ 
admits a family of regular solutions 
{\em locally, i.e. for each $u$ 
in an open neighborhood of the nominal parameter instance
      $\langle u \rangle \in X$}.
In particular, for all $u$ in a sufficiently small, 
closed neighborhood $\tilde{X}\subseteq X$ of $\langle u \rangle\in X$ 
(such as $ {\mathscr B}_X(\langle u \rangle;R)$ in \eqref{eq:NonOpEqn})
the problem \eqref{eq:main} is well-posed 
(see, e.g., \cite[Chapter~IV.3]{GR90}, or \cite{BRR,PR}):
for every fixed $u\in \tilde{X}\subset X$,
and for every $F(u)\in \bcY'$, 
there exists a unique solution $q(u)$ 
of \eqref{eq:main} which depends continuously on $u$.

As in \cite{BRR}, we call the set 
$\{(u,q(u)): u\in \tilde{X} \} \subset X\times \bcX$ a
{\em regular branch of solutions} of \eqref{eq:main} if
\begin{equation}\label{eq:RegBranch}
\begin{array}{l}
\tilde{X}\ni u \mapsto q(u) \;\mbox{is continuous as mapping from} \; 
\tilde{X} \mapsto \bcX\; \mbox{and}
\\
\cR(u;q(u)) = 0 \quad \mbox{in}\quad \bcY' \mbox{ for all } u \in \tilde{X}
\;.
\end{array}
\end{equation}
We call \eqref{eq:RegBranch} {\em branch of nonsingular solutions} 
if, in addition to \eqref{eq:RegBranch}, the differential
\begin{equation}\label{eq:NonSingBranch}
(D_q\cR)(u;q(u))\in \cL(\bcX,\bcY') \;
\mbox{is an isomorphism from $\bcX$ onto $\bcY'$,
     for all $u\in \tilde{X}$}
\;.
\end{equation}
The following proposition 
collects well-known sufficient conditions 
for well-posedness of \eqref{eq:main}.
For regular branches of nonsingular solutions 
given by \eqref{eq:main} - \eqref{eq:NonSingBranch},
the differential $D_q\cR$ satisfies the so-called 
{\em inf-sup conditions}.
\begin{proposition}\label{prop:WellposInfSup}
Assume that $\bcY$ is reflexive and that,
for some nominal value $\langle u \rangle\in X$ 
of the uncertainty, the operator equation 
\eqref{eq:main} admits a {\em regular branch of nonsingular solutions}
\eqref{eq:RegBranch}, \eqref{eq:NonSingBranch}.
Then the differential 
$D_q\cR \in \cL(\bcX,\bcY')$ 
at 
$(\langle u \rangle,q_0) \in X\times \bcX$,
given by the bilinear map 
$$\bcX \times \bcY \ni (\varphi,\psi) 
 \mapsto \,_{\bcY'}\langle  
 D_q\cR(\langle u\rangle;q_0)\varphi,\psi\rangle_\bcY
\;,
$$
is boundedly invertible, 
{\em uniformly with respect to $u\in \tilde{X}$}
where $\tilde{X} \subset X$ is an open neighborhood
of the nominal instance $\langle u \rangle \in X$ of the uncertain
parameter
if and only if there exists a constant $0 < \mu \leq 1$
such that there holds
\be\label{eq:DqAinfsup}
\forall u\in \tilde{X}: \quad 
\begin{array}{c} 
\displaystyle
\inf_{0\ne\varphi\in \bcX} \sup_{0\ne\psi   \in \bcY}
\frac{
 _{\bcY'}\langle  (D_q \cR)(u;q_0)\varphi,\psi\rangle_\bcY
}{
\| \varphi \|_\bcX \|\psi \|_{\bcY}
}
\geq \mu > 0 \;,
\\
\displaystyle
\inf_{0\ne\psi\in \bcY} \sup_{0\ne\varphi \in \bcX}
\frac{
 _{\bcY'}\langle  (D_q \cR)(u;q_0)\varphi , \psi \rangle_\bcY
}{
\| \varphi \|_\bcX \|\psi \|_{\bcY}
}
\geq \mu > 0
\end{array}
\ee
and
\be\label{eq:DqAsupsup}
\forall u \in \tilde{X}:
\quad 
\| (D_q\cR)(u;q_0) \|_{\cL(\bcX,\bcY')} 
=
\sup_{0\ne\varphi\in \bcX} 
\sup_{0\ne\psi   \in \bcY}
\frac{
_{\bcY'}\langle (D_q\cR)(u;q_0)\varphi,\psi\rangle_\bcY
}{
\| \varphi \|_\bcX \|\psi \|_{\bcY}
}
\leq \mu^{-1} \;.
\ee
\end{proposition}
Under conditions \eqref{eq:DqAinfsup} and \eqref{eq:DqAsupsup}, 
for every $u\in \tilde{X} \subseteq X$, 
there exists a unique, regular solution $q(u)$ of \eqref{eq:main}
which is uniformly bounded with respect to $u\in \tilde{X}$
in the sense that there exists
a constant $C(F,\tilde{X}) > 0$, independent of $q$, 
such that 
\begin{equation}\label{eq:LocBdd}
\sup_{u\in \tilde{X}} \| q(u) \|_{\bcX} \leq C(F,\tilde{X})
\;.
\end{equation}
For \eqref{eq:DqAinfsup} - \eqref{eq:LocBdd} being valid, 
we shall say that the set 
$\{ (u,q(u)): u\in \tilde{X}\}\subset \tilde{X}\times \bcX$
forms a {\em regular branch of nonsingular solutions}.

If, in addition to Frechet differentiability 
of $\cR$ with respect to $q$, for every $u\in \tilde{X} \subseteq X$,
the nonlinear functional is also Frechet differentiable
with respect to $u$ at every point of the regular branch
    $\{ (u,q(u)): u\in \tilde{X}\}\subset \tilde{X}\times \bcX$,
then the dependence of the mapping relating $u$ to $q(u)$ with the branch
of nonsingular solutions, is locally Lipschitz on ${\tilde{X}}$: 
i.e.
there exists a Lipschitz constant $L(F,\tilde{X})$ 
such that
\begin{equation}\label{eq:LocLip}
\forall u,v\in \tilde{X}:
\quad 
\| q(u) - q(v) \|_{\bcX} \leq L(F,\tilde{X}) \| u-v \|_X 
\;.
\end{equation}
This follows from $(D_u q)(u) = - (D_q\cR)^{-1} ( D_u\cR)$, and from the bounded invertability of the differential $D_q\cR$ on the regular branch, implied by \eqref{eq:DqAinfsup}.
 
In what follows, we place ourselves in the abstract
setting \eqref{eq:main} with a uniformly continuously differentiable
mapping $\cR(u;q)$ in a product of neighborhoods
${\mathscr B}_X(\langle u \rangle;R)\times {\mathscr B}_\bcX(q(\langle u \rangle);R)$
of sufficiently small radius $R>0$.
The quantity
$q(\langle u \rangle) \in \bcX$ is the corresponding 
regular solution of \eqref{eq:main} at 
the nominal value $\langle u \rangle\in X$.
\subsection{Uncertainty parametrization}
\label{sec:Param}
We shall be concerned with the particular
case where $u\in X$ is a random
variable taking values in a subset $\tilde{X}$ of 
the Banach space $X$.
We assume that $X$ is separable, infinite-dimensional, 
and admits an unconditional Schauder basis $\{\psi_j\}_{j\geq 1}$:
$X = {\rm span}\{\psi_j: j\geq 1\}$.
Moreover, we assume the {\em summability condition}
\begin{equation}\label{eq:psumpsi0}
\sum_{j\geq 1} \| \psi_j \|_X <\infty\;.
\end{equation}
Let $U = [-1,1]^{\mathbb{N}}$ 
and assume that\footnote{In our discussion of QMC quadrature 
ahead, we rescale this set to $[-1/2,1/2]^\bbN$, shift it to $[0,1]^\bbN$
in order to 
integrate with respect to the product of the Lebesgue-measure in $[0,1]^{\mathbb{N}}$.}
\begin{equation*}
\tilde{X} 
= 
\left\{u = \langle u \rangle + \sum_{j \ge 1} y_j \psi_j: \bsy = (y_1, y_2, \ldots ) \in U \right\}
\;.
\end{equation*}
The properties of the set $\tilde{X}$ depend on the properties of 
the sequence $(\psi_j)_{j \ge 1}$. 
Uncertain data $u$ with ``higher regularity'' (when measured in 
a smoothness scale $\{ X_t \}_{t\geq 0}$ with 
$X=X_0 \supset X_1 \supset X_2 \supset ...$) 
corresponds to a stronger decay of the sequence $(\| \psi_j \|_X)_{j \ge 1}$: 
specifically, 
we shall assume in what follows that $\{\psi_j\}_{j\geq 1}$
is scaled such that the sequence $\bsb = (b_j)_{j\geq 1}$ 
given by
\begin{equation}\label{eq:psumpsi}
\bsb := \{ \| \psi_j \|_X \}_{j\geq 1} \in \ell^p(\N) \;\;\mbox{for some} \;\;0< p < 1
\;.
\end{equation}
Once an unconditional Schauder basis $\Psi = \{\psi_j\}_{j\geq 1}$ 
of $X$ has been selected, 
every realization $u\in \tilde{X}$ can be identified in 
a one-to-one fashion with the pair $(\langle u \rangle,\bsy)$ 
via
\begin{equation}\label{eq:uviapsi}
u = \langle u \rangle + \sum_{j \ge 1} y_j \psi_j,
\end{equation}
where $\langle u \rangle\in X$ denotes the 
{\em nominal instance} of the uncertain datum $u$ and 
$\bsy$ is the coordinate vector
of the basis representation \eqref{eq:uviapsi}.
\begin{remark}\label{remk:U=X}
The operator $A(u;q)$ in \eqref{eq:main} becomes,
via the uncertainty parametrization \eqref{eq:uviapsi},
a parametric, deterministic operator family $A(u(\bsy);q)$ 
which we denote (with a slight abuse of notation)
by $\{ A(\bsy;q) : \bsy \in U\}$, with
the parameter set $U=[-1,1]^\bbN$ and with the basis $\{\psi_j\}_{j\geq 1}$.
Similarly we write $\cR(\bsy;q)$ instead of $\cR(u; q)$ in the following.
In the particular case that the parametric 
operator family is linear, we have 
$A(\bsy;q) = A(\bsy)q$ with $A(\bsy)\in \cL(\cX,\cY')$.
We do not assume, however, that the maps $\cX \ni q\mapsto A(\bsy;q) \in \cY'$
are linear.
\end{remark}
With these conventions and with \eqref{eq:uviapsi}, 
we may restate \eqref{eq:main} as parametric operator equation: 
given $F:U\rightarrow \cY'$,
\be\label{eq:paraOpEq} 
\mbox{find}\; q(\bsy;F)\in \cX:\;
\forall \bsy\in U:\;\;
\cR(\bsy;q) := A(\bsy;q) - F(\bsy) = 0 \quad \mbox{in}\quad \cY'
\ee
or, equivalently, with 
$\fa(\bsy;q,v) \;= \; {_{\cY'}\langle A(\bsy;q),v\rangle_{\cY} }$
and 
$\ff(\bsy;v)\; :=\; _{\cY'}\langle F(\bsy),v\rangle_{\cY}$,
\be\label{eq:paraOpEqw}
_{\cY'}\langle \cR(\bsy;q), v \rangle_\cY 
= 
\fa(\bsy;q,v) - \ff(\bsy;v) = 0 
\quad \forall v\in \cY\;.
\ee
With this understanding, and under the assumptions
\eqref{eq:LocBdd} and \eqref{eq:LocLip}, 
the operator equation
\eqref{eq:main} will admit, for every $\bsy\in U$, 
a unique solution $q(\bsy;F)$ which is, due to
\eqref{eq:LocBdd} and \eqref{eq:LocLip}, 
uniformly bounded and depends Lipschitz continuously
on the parameter sequence $\bsy\in U$: 
there holds
\begin{equation}\label{eq:ParDepy}
\sup_{\bsy\in U} \| q(\bsy;F) \|_{\cX} \leq C(F,U) 
\end{equation}
for some constant $C(F,U)$ which is independent of $q$, 
and, if the local Lipschitz condition \eqref{eq:LocLip} holds, 
there exists a Lipschitz constant $L>0$ 
such that (denoting by $\bsy, \bsy' \in U$ the
coefficient sequences associated with $u,v\in \tilde{X}$
via \eqref{eq:uviapsi})
\begin{equation}\label{eq:ParLip}
\begin{array}{rcl}
\| q(\bsy;F) - q(\bsy';F) \|_\cX 
& \leq & 
L(F,U) 
\| u - v \|_X
\\
& \leq & 
L(F,U) 
\| \bsy - \bsy' \|_{\ell^\infty(\bbN)}
\sum_{j\geq 1} \| \psi_j \|_X 
\;.
\end{array}
\end{equation}
We remark that the Lipschitz constant $L(F,U) > 0$ in \eqref{eq:ParLip} 
is not, in general, equal to $L(F,\tilde{X})$ in \eqref{eq:LocLip}:
it depends on $\langle u\rangle\in X$
and on the choice of $\{\psi_j\}_{j\geq 1}$.
\subsection{(Petrov-)Galerkin discretization}
\label{sec:Discr}
%
%
In this section we present, 
based on the theory in \cite[Chapter~IV.3]{GR90} and
in \cite{PR}, which goes back to \cite{BRR} and to M. Crouzeix,
an error analysis of (Petrov-)Galerkin discretizations of \eqref{eq:paraOpEq}
for the approximation of regular branches of solutions 
of smooth, nonlinear problems \eqref{eq:main}.
This will allow us, in the next section, to 
generalize the results \cite{KSS12,KSS11,KSS13} on 
Quasi-Monte Carlo (QMC) (Petrov-)Galerkin approximations 
for countably-parametric operator equations \eqref{eq:paraOpEq}.

To this end, as in \cite{ScMCQMC12,DKGNS13}, we assume that we are
given two one-parameter sequences
$\{ \cX^h \}_{h>0}\subset\cX$ and $\{ \cY^h \}_{h>0}\subset\cY$ of 
finite dimensional subspaces. 
We assume also that, as the discretization parameter $h\downarrow 0$, 
these sequences are dense in $\cX$ and in $\cY$, respectively. 
For the computational complexity analysis, 
we further assume the following {\em approximation properties}: 
there is a scale $\{ \cX_t \}_{t\geq 0}$ of subspaces such that 
$\cX_{t'} \subset \cX_t \subset \cX_0 = \cX$ for any 
$0<t<t'<\infty$ and such that,
for $0 < t \leq \bar{t}$ 
and for $0< h \leq h_0$, there holds
\begin{equation} \label{eq:apprprop}
\begin{aligned}
\forall v \in \cX_t\;
&:
\quad
\inf_{v^h\in \cX^h} \| v - v^h \|_{\cX}
\,\leq\,
C_t\, h^t\, \| v \|_{\cX_t} \;.
\end{aligned}
\end{equation}
Typical examples of smoothness scales $\{\bcX_t \}_{t\geq 0}$ are
given by the Sobolev scale $\bcX_t = H^{1+t}(D)$ 
in smooth domains 
(or by its weighted counterparts in polyhedra \cite{NS12}).
\begin{proposition}\label{prop:stab}
Under the assumptions of Proposition~\ref{prop:WellposInfSup} and if, in addition,  the subspace sequences 
$\{ \cX^h \}_{h>0}\subset\cX$ and $\{ \cY^h \}_{h>0}\subset\cY$ are stable, 
i.e., there exist $\bar{\mu} > 0$
and $h_0 > 0$ such that for every $0<h \leq h_0$, there 
hold the uniform (with respect to $\bsy\in U$) 
discrete inf-sup conditions
\begin{align}\label{eq:Bhinfsup1}
&\forall \bsy \in U:
\quad
\inf_{0\ne v^h \in \bcX^h} \sup_{0\ne w^h \in \bcY^h}
\frac{
 _{\bcY'}\langle  (D_q\cR)(\bsy;q_0)v^h,w^h\rangle_\bcY
}{
\| v^h \|_\bcX \| w^h \|_{\bcY}
}
\geq \bar{\mu} > 0\;,
\\
\label{eq:Bhinfsup2}
&\forall \bsy\in U:\quad
\inf_{0\ne w^h \in \bcY^h} \sup_{0\ne v^h \in \bcX^h}
\frac{_{\bcY'}\langle  (D_q\cR)(\bsy;q_0)v^h,w^h\rangle_\bcY}{\| v^h \|_{\bcX} \|w^h\|_{\bcY}}
\geq \bar{\mu}>0
\;.
\end{align}
Assume in addition that the differential $ (D_q\cR)(\bsy ;q) \in \cL(\cX,\cY')$ 
is Lipschitz with respect to $q$, uniformly with respect to 
$\bsy\in U$, 
i.e.
\be\label{eq:LipDqR}
\forall \bsy\in U:\quad 
\| (D_q\cR)(\bsy;q) - (D_q\cR)(\bsy;\tilde{q}) \|_{\cL(\cX,\cY')} 
\leq 
L \| q - \tilde{q} \|_{\cX} \;,
\ee
where the Lipschitz constant is independent of $\bsy$.

Then, for every $0<h \leq h_0$ 
the (Petrov-)Galerkin approximations: 
given $\bsy\in U$,
\begin{equation} \label{eq:parmOpEqh}
\mbox{find} \; q^h(\bsy) \in \bcX^h :
\quad
{_{\bcY'}}\langle \cR(\bsy;q^h(\bsy)), w^h \rangle_{\bcY} = 0
\quad 
\forall w^h\in \bcY^h\;,
\end{equation}
are uniquely defined and converge quasioptimally; i.e.
there exists a constant $C>0$ such that for all $\bsy\in U$
\begin{equation} \label{eq:quasiopt}
 \| q(\bsy) - q^h(\bsy) \|_{\cX}
 \,\le\, 
\frac{C}{\bar{\mu}} \inf_{0\ne v^h\in \cX^h} \| q(\bsy) - v^h\|_{\cX}
\;.
\end{equation}
If the parametric response $q(\bsy)$ belongs to $\bcX_t$ 
uniformly w.r.t. $\bsy\in U$,
and if, moreover, \eqref{eq:apprprop} holds, then 
there exists a constant $C>0$ such that, for every $\bsy\in U$
\begin{equation}\label{eq:convrate}
\| q(\bsy) - q^h(\bsy) \|_{\cX}
\,\le\, 
\frac{C}{\bar{\mu}} h^t \sup_{\bsy\in U} \| q(\bsy) \|_{\bcX_t} 
\;.
\end{equation}
\end{proposition}
This result follows, under the stated hypotheses, from \cite[Theorem~4]{PR}.
In the ensuing QMC convergence analysis we shall also require
error bounds for the dimensionally truncated parameter sequences.
The present framework of regular branches of solutions
of nonlinear operator equations covers many equations of interest
in applications: we mention only problems of viscous, incompressible
flows (see, e.g., \cite[Chapter~IV.4, IV.5]{GR90} for viscous,
incompressible flow, \cite[Section~5]{CCS2} for nonlinear, elliptic
PDEs in uncertain domains, and for linear, parabolic PDEs in 
uncertain domains).
\subsection{Dimension truncation}
\label{sec:dimtrunc}
For a {\em truncation dimension} $s\in \N$, 
denote the $s$-term truncation of the series representation \eqref{eq:uviapsi} of 
the uncertain datum $u$ by $u^s\in X$.
Then, dimension truncation is equivalent to setting $y_j=0$ 
for $j>s$ in \eqref{eq:uviapsi}. \jd{For $\bsy\in U$, we define 
\be \label{eq:y1s}
\bsy_{\{1:s\}} := (y_1,y_2,...,y_s,0,0,...).
\ee }
We denote by $q^s(\bsy)$ the solution of the corresponding parametric weak problem \eqref{eq:paraOpEq}. 
Unique solvability of \eqref{eq:paraOpEq} 
for every $\bsy\in U$ implies also unique solvability for the dimension truncated problem
with solution $q^s(\bsy) = q( \jd{ \bsy_{\{1:s\}} } )$ 
and introduce $u^s(\bsy) := u(\bsy_{\{1:s\}})$.
We bound the {\em dimension truncation error}
$q(\bsy) - q^s(\bsy)$ based on
\begin{assumption}\label{asmp:psum}
(i) 
$\bsb\in \ell^p(\IN)$ for some $0<p<1$, i.e. \eqref{eq:psumpsi} holds;

\noindent
(ii) the $b_j$ are enumerated in non-increasing order, i.e.
\begin{equation} \label{eq:ordered} 
b_1 \ge b_2 \ge \cdots \ge b_j \ge \, \cdots\;.
\end{equation}
\end{assumption}
Under Assumption \ref{asmp:psum},
we consider the {\em $s$-term truncated problem}: 
given $\bsy_{\{1:s\}} \in U$ 
\begin{equation}\label{eq:mainstrunc}
\mbox{find}\;q^s\in \bcX:\quad 
{ _{\bcY'} \langle \cR(\bsy_{\{1:s\}};q^s), w \rangle_{\bcY} } = 0 
\;\;\forall w \in \bcY 
\;.
\end{equation}
Under our assumption on well-posedness of the problem
\eqref{eq:NonOpEqn} uniformly for all 
$u\in {\mathscr B}_X(\langle u \rangle;R)$,
the basis property \eqref{eq:uviapsi} of the sequence $\{ \psi_j \}$
implies that $u^s \in \tilde{X}$.
Therefore, the parametric problem \eqref{eq:mainstrunc} 
admits a unique solution for every $\bsy\in U$.
\begin{theorem} \label{thm:trunc}
Under the Assumptions in Section \ref{sec:OpEqUncInp},
and assuming \eqref{eq:psumpsi} and \eqref{eq:uviapsi},
for every $f\in \cY'$, for every
$\bsy\in U$ and for every $s\in\bbN$, 
the solution $q^s(\bsy)$ of the 
parametric weak problem \eqref{eq:paraOpEq} 
with $s$-term truncated parametric expansion \eqref{eq:uviapsi}
satisfies, with $\bsb$ as defined in 
\eqref{eq:psumpsi},
\begin{equation}\label{eq:Vdimtrunc}
\sup_{\bsy\in U}
  \| q(\bsy) - q^s(\bsy) \|_\bcX
  \,\le\, C(F,X)
  \sum_{j\ge s+1} b_j
\;.
\end{equation}
Moreover, for every observation functional 
$\LinFun(\cdot)\in \cX'$, there holds the 
dimension-truncation error bound
\begin{equation}\label{eq:Idimtrunc}
  |I(\LinFun(q))- I(\LinFun(q^s))|
  \,\le\, \tilde{C} 
  \bigg(\sum_{j\ge s+1} b_j\bigg)
\end{equation}
for some constant $\tilde{C}>0$ independent of $s$.
In addition, 
if conditions~\eqref{eq:psumpsi0}, \eqref{eq:psumpsi} 
and \eqref{eq:ordered} hold, then
\begin{equation}\label{eq:DTbound}
  \sum_{j\ge s+1} b_j
  \,\le\,
  \min\left(\frac{1}{1/p-1},1\right)
  \bigg(\sum_{j\ge1} b_j^p \bigg)^{1/p}
  s^{-(1/p-1)}
  \;.
\end{equation}
\end{theorem}
\begin{proof}
From the Lipschitz dependence \eqref{eq:LocLip}, 
we obtain
$$
\forall \bsy\in U:\quad 
\| q(\bsy) - q^s(\bsy) \|_{\bcX} \leq L(F,\tilde{X}) \| u(\bsy) - u^s(\bsy) \|_X
\;.
$$
From  \eqref{eq:uviapsi}, the $p$-summability \eqref{eq:psumpsi} 
of the sequence $\bsb$ and the
monotonicity \eqref{eq:ordered} we infer that the error of the 
$s$-term truncation $u^s(\bsy)$, $\| u(\bsy) - u^s(\bsy)\|_X$, 
can be bounded by a best $s$-term truncation error of 
$\bsb$ in the norm of $\ell^1(\IN)$ by
$$
\sup_{\bsy\in U} 
\| u(\bsy) - u^s(\bsy)\|_X
\leq 
\sum_{j \geq s+1}
b_j
\;.
$$
The $p$-summability $\bsb\in \ell^p(\IN)$ 
in Assumption \ref{asmp:psum}(i) and the 
(assumed) ordering \eqref{eq:ordered} imply \eqref{eq:DTbound}.
\end{proof}
As $\bsy\in U$ implies $\bsy_{\{1:s\}} \in U$ for all
$s\in \IN$, we obtain from  Proposition \ref{prop:stab}
immediately
\begin{corollary}\label{coro:DimTrcStab}
Under the assumptions of  Proposition \ref{prop:stab},
for given $\bsy_{\{1:s\}}\in U$,
the dimensionally truncated (Petrov-)Galerkin approximations 
\begin{equation} \label{eq:parmOpEqh_trun}
\mbox{find} \; q^h(\bsy_{\{1:s\}}) \in \bcX^h :
\quad
{{_{\bcY'}}\langle  \cR(\bsy_{\{1:s\}};q^h(\bsy_{\{1:s\}})), w^h \rangle_{\bcY} } = 0
\quad 
\forall w^h\in \bcY^h
\;,
\end{equation}
admit unique solutions $q^h(\bsy_{\{1:s\}}) \in \bcX^h$ which
converge, as $h\downarrow 0$, quasioptimally to $q(\bsy_{\{1:s\}}) \in \bcX$, 
i.e. \eqref{eq:quasiopt} and \eqref{eq:convrate} hold 
with $\bsy_{\{1:s\}}$ in place of $\bsy$, with
$C>0$ and $\bar{\mu}>0$ independent of $s$, of $\bsy\in U$ and of $h$.
\end{corollary}
%
\subsection{Holomorphic parameter dependence}
\label{sec:Hol}
%
In the error analysis for QMC integration methods
as presented, e.g., in \cite{KSS12,KSS11,KSS13},
derivative bounds for the integrand functions 
that are explicit with respect to the dimension $s$
are essential. 
In \cite{CDS2,CCS2}, such bounds were obtained
via holomorphy of countably parametric families of
operator equations and their parametric solutions.
By this we mean that the parametric family of solutions
permits, with respect to each parameter $y_j$, a holomorphic
extension into the complex domain $\C$; for purposes of QMC
integration, in addition, some uniform bounds on these
holomorphic extensions must be satisfied in order to 
prove approximation rates and QMC quadrature error bounds
which are independent of the number of parameters 
which are ``activated'' 
in the QMC quadrature process.

In the remainder of Section \ref{sec:HolOpEq} and 
throughout the next Section \ref{sec:anadepsol}, 
\emph{
all
spaces $X$, $\bcX$ and $\bcY$ will be understood as 
Banach spaces over $\C$, without notationally indicating
so.
}
\subsubsection{$(\bsb,p,\eps)$-Holomorphy}
\label{sec:bpeHol}
In \cite{HaSc11,CCS2}, the notion of 
{\em $(\bsb, p, \eps)$-holomorphy of parametric solutions}
has been introduced. 
For $\kappa>1$, we define the sets $\cT_\kappa\supset [-1,1]$ as
\begin{equation}\label{eq:TubeDef}
\cT_\kappa 
= \{ z\in \bbC \mid {\rm dist}(z,[-1,1]) \leq \kappa - 1\}
= \bigcup_{-1\leq y \leq 1} \{ z\in \IC \mid |z-y|\leq \kappa -1 \} 
\subset \bbC
\;.
\end{equation}
\begin{definition} \label{def:peanalytic} ($(\bsb,p,\eps)$-holomorphy)
For $\eps > 0$ and for a positive sequence 
$\bsb = (b_j)_{j\geq 1} \in \ell^p(\IN)$ for
some $0<p<1$, we say that a parametric solution family 
$q(\bsy) : U\mapsto \cX$ of \eqref{eq:main}
satisfies the {\em $(\bsb,p,\eps)$-holomorphy assumption} 
if and only if all of the following conditions hold:
\begin{enumerate}
\item 
For each $\bsy\in U$, 
the map $\bsy\mapsto q(\bsy)$ from $U$ to $\cX$
is uniformly bounded w.r.t. the parameter sequence $\bsy$, 
i.e.
\be
\label{ubNewProblem}
\sup_{\bsy\in U}\|q(\bsy)\|_X \leq B_0\;,
\ee
for some finite constant $B_0>0$.
\item 
For any sequence $\bsrho :=(\rho_j)_{j\geq 1}$
of numbers $\rho_j > 1$ that satisfies 
\be
\label{eq:rho_b}
\sum_{j \geq 1} (\rho_j-1) b_j \leq \eps,
\ee
%
the parametric solution map $U\ni\bsy\mapsto q(\bsy)$ 
admits an extension
$\bsz \mapsto q(\bsz)$ to the complex domain
that is holomorphic with
respect to each variable $z_j$ 
in a cylindrical set of the form 
$\cO_{\bsrho} := \bigotimes_{j\geq 1} \cO_{\rho_j}$,
where, for every integer $j\geq 1$, 
$\cO_{\rho_j}\subset\C$ 
is an open set containing the closed tube $\cT_{\rho_j}$.
For a poly-radius $\bsrho$ satisfying
\eqref{eq:rho_b}, we denote by $\cT_\bsrho$ 
the corresponding cylindrical set 
$\cT_\bsrho := \bigotimes_{j\geq 1} \cT_{\rho_j}\subset \C^\N$.

\item
For any poly-radius $\bsrho $ satisfying \eqref{eq:rho_b}, 
there is a second family 
$\tilde{\cO}_{\bsrho} := \bigotimes_{j\geq 1} \tilde{\cO}_{\rho_j}$
of open, cylindrical sets 
$$
{\cO}_{\rho_j} \subset \tilde{\cO}_{\rho_j} \subset \C
$$
(strict inclusions), such that the
extension is bounded on the closure 
$\overline{\tilde{\cO}_\bsrho}$ of $\tilde{\cO}_{\bsrho}$
according to
\be
\sup_{\bsz\in\tilde{\cO}_\rho}\|q(\bsz)\|_X \leq B_\eps \, ,
\ee
where $B_\eps > 0$ depends on $\eps$, 
but is independent of $\bsrho$.
\end{enumerate}
\end{definition}

The notion of $(\bsb, p,\eps)$-holomorphy depends implicitly on the 
choice of sets $\cO_\rho$ and $\tilde{\cO}_{\rho}$.
Depending on the approximation process in the parameter 
domain $U$ under consideration, a particular choice of the sets 
$ \tilde{\cO}_{\rho_j}$ has to be made in order to obtain sharp
convergence bounds under minimal holomorphy requirements.

In \cite{CCS2,CDS2},
the sets $\cO_\rho$ were chosen to contain
Bernstein ellipses $\cE_\rho$ which are natural in the context
of Legendre polynomial chaos approximations. In the context
of Taylor- or Tschebyscheff polynomial approximations,
polydiscs $\cO_\rho$ are natural (cf. \cite{HaSc11}).
For the derivative bounds which arise in connection
with higher order QMC error analysis
(see, e.g., \cite{DKGNS13,KSS12}),
we use the tubes $\cT_{\rho}$ 
{
\eqref{eq:TubeDef} as continuation domains
$\cO_\rho$ and 
$\tilde{\cO} = \cT_{\tilde{\rho}}$ with $\tilde{\rho} > \rho > 1$.
}
%
\subsubsection{Holomorphic parametric operator equations}
\label{sec:ParOpEq}
We next consider parametric models \eqref{eq:paraOpEq}
and the regularity of their (countably-) parametric solution
families.
The following result, \cite[Theorem~2.4]{CCS2},
ensures $(\bsb, p,\eps)$-holomorphy 
of the parametric solution map $\bsy\mapsto q(\bsy)$ 
with respect to the holomorphy domains $\cT_{\bsrho}$ 
in Definition \ref{def:peanalytic}
under the assumption of $(\bsb, p,\eps)$-holomorphy of the 
parametric maps $A$ and $F$ in \eqref{eq:main} 
and \eqref{eq:paraOpEq}.
\begin{theorem} \label{thm:AnalyticLaxMilgram}
Assume that in \eqref{eq:psumpsi} it holds $\bsb\in \ell^p(\IN)$ for
some $0< p \leq 1$. 
Assume further that (recall that $X$, $\bcX$ and $\bcY$ are
understood as Banach spaces over $\IC$) the residual map 
$X\times \bcX \ni (u,q) \mapsto \cR(u;q) \in \bcY'$ 
in \eqref{eq:NonOpEqn} is continuously Frechet-differentiable, 
and
$$
\forall u\in \tilde{X}: \quad (D_q\cR)(u; q(u)) \in \jd{\cL(\bcX, \bcY') } 
$$
\jd{is an isomorphism.}

Then there holds: 
(i) 
The parametric residual map 
$\cR(\bsz;q)$ in \eqref{eq:main}, \eqref{eq:paraOpEqw} 
admits a holomorphic extension (still denoted by $\cR(\bsz;q)$)
which satisfies the $(\bsb, p,\eps)$-holomorphy assumptions 
for $\bsz \in \cT_{\bsrho}$ with the same $p$ and $\eps$ 
and with the same sequence $\bsb$.

(ii) 
Then there exists $\eps > 0$ such that
the parametric regular branch of nonsingular solutions,
$ U \ni \bsy \mapsto q(\bsy)$,
admits a holomorphic extension with respect to the parameters
$\bsy$ to the sets 
$\cT_{\bsrho} = \bigotimes_{j\geq 1}\cT_{\rho_j}$ 
with $\cT_{\rho_j}$ as in \eqref{eq:TubeDef}, for any 
$\bsrho = (\rho_j)_{j\geq 1}$ which satisfies \eqref{eq:rho_b}.
In particular, the parameter dependence of 
this holomorphic extension $\bsz\mapsto q(\bsz) \in \bcX$ 
of the regular branch
of solutions is $(\bsb,p,\eps)$-holomorphic.
\end{theorem}
%
\section{Parametric regularity of solutions}
\label{sec:anadepsol}
In this section we study the dependence of the solution $q(\bsy)$ of
the parametric, variational problem \eqref{eq:paraOpEq} on the parameter
vector $\bsy$, with precise bounds on the growth of the partial
derivatives. These derivative bounds imply, 
in conjunction with the results in \cite{KSS12}, 
dimension independent convergence rates for QMC quadratures.

In the following, let $\N_0^\N$ denote the set of sequences $\bsnu =
(\nu_j)_{j\geq 1}$ of nonnegative integers $\nu_j$, and let $|\bsnu| :=
\sum_{j\geq 1} \nu_j$. For $|\bsnu|<\infty$, 
we denote the partial
derivative of order $\bsnu$ of $q(\bsy)$ 
with respect to $\bsy$ by
\be\label{eq:dqdy}
\partial^\bsnu_\bsy q(\bsy)
\,:=\,
\frac{\partial^{|\bsnu|}}{\partial^{\nu_1}_{y_1}\partial^{\nu_2}_{y_2}\cdots}q(\bsy)
\;.
\ee
In \cite{CDS1,KSS12,KunothCS2011}, bounds on the derivatives \eqref{eq:dqdy} 
were obtained by an induction argument which strongly relied on 
affine-parametric dependence of $A(\bsy;q)$ on $\bsy$.

Here, we derive 
alternative bounds on $\| (\partial^{\bsnu}_\bsy q)(\bsy)\|_{\cX}$
based on complex variable methods 
which were used also in \cite{CDS2,ScSt11,SS12,SS13,CCS2}.
We shall see that in QMC integration these bounds 
give rise to product weights
at least for a finite (possibly large, but in general operator-dependent)
``leading'' dimension of the parameter space.
The argument is based on 
\emph{holomorphic extension of the parametric integrand functions
into the complex domain} 
(we remark that not all PDE problems afford such 
 extensions and refer to \cite{HoSc12Wave} for an example).
 
In certain cases, the possibility of covering the parameter intervals $[-1,1]$
by a finite number of small balls (whose union is contained
in a tube $\cT_{\rho_j}$ \eqref{eq:TubeDef} 
for a radius $\rho_j >1 $ sufficiently close to $1$) 
is required to verify $(\bsb, p,\eps)$-holomorphy 
for certain nonlinear operator equations, 
see for example \cite[Lemma~5.2]{CCS2}. 
\begin{theorem}\label{thm:DsiboundC}
For every mapping $q(\bsy):U\mapsto \cX$ 
which is $(\bsb, p,\eps)$-{holomorphic} on a polytube $\cT_{\bsrho}$
of
poly-radius $\bsrho = (\rho_j)_{j\geq 1}$ with 
$\rho_j > 1$ satisfying \eqref{eq:rho_b},
there exists a sequence $\bsbeta\in \ell^p(\bbN)$ 
(depending on the sequence $\bsb$ in \eqref{eq:rho_b})
and a partition $\bbN = E \cup E^c$ such that 
the parametric solution $q(\bsy)$ satisfies,
for every $\bsnu\in \bbN_0^\bbN$ with $|\bsnu|<\infty$,
the bound
\be \label{eq:HybBd}
\sup_{\bsy \in U}
\| (\partial^\nu_\bsy q)(\bsy) \|_\cX
\leq
C
\bsnu_E! 
\prod_{j\in E}\beta_j^{\nu_j}  
\times 
|\bsnu_{E^c}|! 
\prod_{j\in E^c}\beta_j^{\nu_j}
\;.
\ee
Here, $E = \{1,2,...,J\}$ for some $J(\bsb)<\infty$ depending on the 
sequence $\bsb$ in \eqref{eq:rho_b}, and for $\bsnu\in \bbN_0^\bbN$,
we set $\bsnu_E := \{ \nu_j : j\in E \}$.
The sequence 
$\bsbeta = (\beta_j)_{j\geq 1}$ satisfies
$\beta_j = 4\| \bsb \|_{\ell^1(\bbN)}/\eps$ for $1 \le j \le J$, 
i.e. 
it is in particular independent of $j$ for $1\leq j \leq J$.
Moreover, 
$\beta_j \lesssim b_j $ for $j>J$ with the implied
constant depending only on $J(\bsb)$ and 
on $\| \bsb \|_{\ell^1(\bbN)}$.
\end{theorem}

The proof of the derivative bound is divided into two steps.
To simplify the notation, we give it for
$|y_j|\leq 1$ and for a poly-radius $\bsrho$ which satisfies
$\rho_j > 1$. 
Later, in Section~\ref{sec:QMC} it is natural to 
consider the parameter domain $[-1/2, 1/2]^{\mathbb{N}}$. The 
assertion for the parameter domain $U = [-1/2,1/2]^\bbN$
then follows via scaling by a factor of $1/2$ 
(see Remark~\ref{rem_interval_change} for details).

In the first step, 
we infer from $(\bsb, p,\eps)$-holomorphy of $q(\bsy)$,
via Cauchy's integral formula, bounds on
$\sup_{\bsy\in [-1,1]^\bbN} \| (\partial^\bsnu_\bsy q )(\bsy)\|_{\cX}$
in terms of the maximum of the analytic continuation
of $q(\bsy)$ to the domain $\cT_\bsrho$ 
of points in the ``polytube'' $\cT_\bsrho$. 
These derivative bounds are valid
{\em for any poly-radius $\bsrho$ which is $(\bsb, p,\eps)$-admissible
      in the sense that \eqref{eq:rho_b} holds.}
The result of the first step is recorded in Lemma~\ref{lem:supq}.

In the second step of the proof, we use a $\bsnu$-dependent 
choice of a $(\bsb, p,\eps)$-admissible poly-radius $\bsrho$ 
for which \eqref{eq:rho_b} holds to obtain the 
($\bsnu$-{\em in}dependent)
weight sequence $\bsbeta$:
for given $\bsnu\in \bbN_0^\bbN$
such that $|\bsnu|<\infty$, we then define a 
$(\bsb, p,\eps)$-admissible poly-radius $\bsrho(\bsnu)$ so that
\eqref{eq:HybBd} is satisfied for this $\bsnu$, with constants
$C_0$ and the sequence $\bsbeta$ independent of $\bsnu$.

In the following let for $\bsnu \in \bbN_0^\bbN$ the support of 
$\bsnu$ be denoted by $ \supp \bsnu := \{ j\in \bbN: \nu_j \ne 0\}\subset \bbN$. 
For a subset $H \subseteq \bbN$, 
we denote its  complement $H^c = \bbN \setminus H$ 
and for a vector $\bsy = (y_j)_{j\ge 1}$,
$\bsy_H = (y_j)_{j \in H}$ denotes its ``restriction'' to $H$.
\begin{lemma}\label{lem:supq}
For every mapping $q(\bsy):U\mapsto \cX$ 
which is $(\bsb, p,\eps)$-holomorphic on a polytube $\cT_{\bsrho}$
of
poly-radius $\bsrho = (\rho_j)_{j\geq 1}$ with 
$\rho_j > 1$ satisfying \eqref{eq:rho_b},
there holds
\[
\sup_{ \bsy_H \in \prod_{j \in H} [-1,1] } 
\| (\partial^\bsnu_\bsy q)(\bsy_H,\bsy_{H^c}) \|_\cX
 \leq \displaystyle 
\sup_{ \bsz_H \in \prod_{j \in H} \cT_{\rho_j} } 
         \| q(\bsz_H,\bsy_{H^c}) \|_\cX
\bsnu!
\prod_{ j \in H }
\frac{\rho_j}{(\rho_j-1)^{\nu_j+1}} \, ,
\]
for every  $\bsnu\in \bbN_0^\bbN$ with $|\bsnu|<\infty$, where
$H = \supp \bsnu$ and for every 
$\bsy_{H^c} \in \prod_{j \in H^c} [-1,1]$.
\end{lemma}
\begin{proof}
The condition \eqref{eq:rho_b} on the poly-radius $\bsrho$ 
implies, with the assumption of $(\bsb,p,\eps)$-holomorphy
of the parametric map $q(\bsy)$, the estimate
\begin{equation}\label{eq-auxiliary-estimate}
\|q(\bsz_H,\bsy_{H^c})\|_{\cX} \leq B
\end{equation}
for some $B \geq 1$ (depending on $\eps$) and
for every $\bsz_H \in \prod_{j \in H} \cT_{\rho_j}$ 
and every $\bsy_{H^c}  \in \prod_{j \in H^c} [-1, 1]$.
To simplify the notation in the following, w.l.o.g. 
we assume that $H=\{1,\ldots,K\}$ for some $K \in \bbN_0$
(this may always be achieved by re-indexing the variables). 
For $\bsb = (b_j)_{j \geq 1}$,
we further define the sequence $\widetilde\bsrho$ by
\[
\widetilde\rho_j=\rho_j+\varepsilon\,,\quad j\in H\,,\quad
\varepsilon
=
\frac{\delta}{\sum_{j\in H} b_j}\,,
\qquad\widetilde\rho_j=\rho_j\,,\quad j \in H^c \,,
\]
for some small real number $\delta>0$.
Then, for $\delta > 0$ sufficiently small, 
also $\widetilde\rho$ is an admissible poly-radius,
in the sense that the parametric solution admits a
holomorphic continuation to the set 
$\cT_{\tilde{\bsrho}}\subset \bbC^\bbN$. In particular, $q_H$ is analytic in an open neighborhood of $U_{\rho,H}$, 
where we are writing 
$q_H(z_1,\ldots,z_K)= q_H(\bsz_H)\equiv q(\bsz_H,0)$.

Cauchy's integral formula can be applied 
successively with respect to each coordinate $z_j\in \cT_{\rho_j}$
with $j\in H$ 
to obtain 
for every $\bsy\in U$ the representation
\[
q(y_1,\ldots,y_K,\bsy_{H^c})
=
(2\pi i)^{-K}\oint_{\Gamma'_1(y_1)}\cdots\oint_{\Gamma'_K(y_K)}
\frac{q(\bsz'_H,\bsy_{H^c})}{(z'_1-y_1)\cdots(z'_K - y_K)}\,
d z'_1\cdots d z'_K\,,
\]
where now $\Gamma'_j(y_j)\subset \bbC$ denotes the circle
with radius $\rho_j-1$ and center $y_j\in [-1,1]$
for $j\in H$. Then, for all $\bsy\in U$,
the integration domains are contained in $\prod_{j \in H} \cT_{\rho_j}$.
Changing the path of integration from $\Gamma'_j(y_j)$ to 
$\partial \cT_{\rho_j}$, the boundary of $\cT_{\rho_j}$,
and differentiating under the integral sign 
in Cauchy's integral formula now yields
for every $\bsy\in U$
\[
(\partial^\bsnu_\bsy q)(\bsy)
=
\bsnu!(2\pi i)^{-K}
\oint_{\partial \cT_{\rho_1}}\cdots\oint_{\partial \cT_{\rho_K}}
\frac{q(z'_1,\ldots,z'_K,\bsy_{H^c})}{(z'_1-y_1)^{\nu_1+1}\cdots (z'_K-y_K)^{\nu_K+1}}
	\,d z'_1\cdots d z'_K
\]
since $\Gamma'_j(y_j)\subset \cT_{\rho_j}$.
Then,
\eqref{eq-auxiliary-estimate}, $|\partial \cT_\kappa| = 2(2+(\kappa-1)\pi)$
and a standard estimate 
for the integral yields the derivative bound: 
for every fixed $\bsy_{H^c}\in U_{H^c}$, 
\be 
\label{eq:DnuyqEst}
\begin{array}{l}
\displaystyle
\sup_{\bsy_H\in \prod_{j \in H}  [-1,1]} 
\| (\partial^\bsnu_\bsy q)(\bsy_H,\bsy_{H^c}) \|_\cX 
\\
\leq \displaystyle 
\sup_{\bsz_H \in \prod_{j \in H} \partial \cT_{\rho_j} }  \| q(\bsz_H,\bsy_{H^c}) \|_\cX
\frac{\bsnu!}{(2\pi)^{K}}
\prod_{ j \in H } |\partial \cT_{\rho_j}| (\rho_j-1)^{-\nu_j-1} 
\\
\leq \displaystyle
\sup_{\bsz_H \in \prod_{j \in H} \cT_{\rho_j}}  \| q(\bsz_H,\bsy_{H^c}) \|_\cX
\frac{\bsnu!}{(2\pi)^{K}}
\prod_{j \in H} 
2(2+(\rho_j-1)\pi) (\rho_j-1)^{-\nu_j-1} \\
= \displaystyle 
\sup_{\bsz_H \in \prod_{j \in H}  \cT_{\rho_j} }  \| q(\bsz_H,\bsy_{H^c}) \|_\cX
\bsnu!
\prod_{ j \in H }
(\frac{2}{\pi}+(\rho_j-1)) (\rho_j-1)^{-(\nu_j+1)}
\\
\leq \displaystyle 
\sup_{\bsz_H \in { \prod_{j \in H} \cT_{\rho_j} } }  \| q(\bsz_H,\bsy_{H^c}) \|_\cX
\bsnu!
\prod_{ j \in H }
\frac{\rho_j}{(\rho_j-1)^{\nu_j+1}}
\;.
\end{array}
\ee
Here we used 
$$
\inf_{y_j \in [-1,1], z_j'\in \partial \cT_{\rho_j}} \{|z_j'-y_j|\} \geq \rho_j - 1 > 0
\;.
$$

\end{proof}
%
\noindent
{\em Proof of Theorem~\ref{thm:DsiboundC}}
In this proof, we will establish \eqref{eq:HybBd} using
the result in Lemma~\ref{lem:supq}.
To obtain these derivative bounds, 
for given $\bsnu\in \IN_0^\IN$ with $|\bsnu|<\infty$
and for  fixed 
$\eps > 0$, we choose, 
with $B\geq 1$ as in \eqref{eq-auxiliary-estimate},
$J = J(\eps,\bsb) \in \N_0$ as
\begin{equation}\label{eq:defS}
J (\eps,\bsb) 
:= 
{\rm min} 
\left\{ s\in \N \mid \sum_{j>s} b_j \leq \frac{\eps}{4B} \leq \frac{\eps}{4} 
\right\}
\;.
\end{equation}
Since $\bsb\in \ell^1(\N)$, \eqref{eq:defS} 
defines for every $\eps > 0$ a unique $J(\eps,\bsb) \in \N$ which, 
as
{\em we emphasize,  is independent of the particular multi-index $\bsnu$.}
With $J = J(\bsb,\eps)$, 
we define the set $E:=\{1,2,...,J\} \subset \N$ 
and define $E^c := \N \backslash E$.
For any multi-index $\bsnu\in \cF$,
we then introduce the
partition $\bsnu = (\bsnu_E,\bsnu_{E^c})$ where
$\bsnu_E := \{ \nu_1,\nu_2,...,\nu_J \}$ and 
$\bsnu_{E^c} := \{ \nu_{J + 1}, \nu_{J + 2},... \}$.
Next, 
we define $\kappa:=1+\eps / (4\| \bsb \|_{\ell^1(\N)}) > 1$
and introduce, for $\bsnu\in \cF$,
the poly-radius $\bsrho(\bsnu)$ by
\be\label{eq:polyrho}
\rho_j := 
\left\{
\begin{array}{ll}
\kappa & \mbox{for } j\in E\;,
\\
\kappa + 
\displaystyle
\frac{\eps}{2b_j} \frac{\nu_j}{1+|\bsnu_{E^c} |} 
& \mbox{for } j\in E^c 
\;.
\end{array}
\right.
\ee
With this choice of $\bsrho(\bsnu)$ we verify 
that \eqref{eq:rho_b} holds. 
This follows since
$$
\begin{array}{rcl}
\displaystyle
\sum_{j\geq 1} 
(\rho_j-1) b_j
&\leq& \displaystyle  
(\kappa-1) \sum_{j=1}^J b_j 
+ 
\sum_{j> J} b_j\left(\kappa -1 + \frac{\eps}{2b_j} \frac{\nu_j}{1+|\bsnu_{E^c} |}\right)
\\
& = &\displaystyle 
\frac{\eps}{4\| \bsb \|_{\ell^1}} \sum_{j=1}^J b_j 
+
\sum_{j > J} b_j
\left(\frac{\eps}{4\| \bsb \|_{\ell^1}} + \frac{\eps}{2b_j} \frac{\nu_j}{1+|\bsnu_{E^c}|}\right)
\\
& \leq &\displaystyle 
\frac{\eps}{4} 
+ 
\frac{\eps}{4}
+
\frac{\eps}{2}
\frac{|\bsnu_{E^c}|}{1+|\bsnu_{E^c} |} \leq \eps
\;.
\end{array}
$$
We introduce the notation
$\phi(\rho) := \frac{\pi}{2} \rho/(\rho-1)$ for $\rho >1$.
The property $\phi'(\rho)<0$ for $\rho >1$ implies,
for $\rho_j$ as in \eqref{eq:polyrho}, that
$\phi(\rho_j) \leq \phi(\kappa)$ for all $j\in \N$. 
Further we have $\phi(\rho) \ge 1$ for all $\rho > 1$.

In the following we prove a bound on $\bsnu! \prod_{j \in H} \frac{\rho_j}{(\rho_j-1)^{\nu_j+1}}$ 
(where $H = \supp \bsnu$), which appears in Lemma~\ref{lem:supq}.
We obtain, assuming 
w.l.o.g. that $J \leq L := \max\{ j: \nu_j > 0 \}$, 
that there holds the bound
$$
\begin{array}{l}
\displaystyle 
\bsnu!
\prod_{ j \in H }
\frac{\rho_j}{(\rho_j-1)^{\nu_j+1}}
= 
\bsnu_E!\bsnu_{E^c} !
\prod_{ 1\leq j \leq L} \frac{2}{\pi} \phi(\rho_j)(\rho_j-1)^{-\nu_j}
\\
\leq \displaystyle 
\bsnu_E! \left\{ \prod_{j\in E} \phi(\kappa) \left(\frac{4\| \bsb \|_{\ell^1}}{\eps}\right)^{\nu_j} \right\}
\times
\bsnu_{E^c}! \left\{ \prod_{ \satop{ j\in E^c }{\nu_j > 0} } 
\phi(\rho_j) \left(\frac{2b_j}{\eps \nu_j}(1+|\bsnu_{E^c}|) \right)^{\nu_j} 
\right\} 
\\
=: \displaystyle 
\bsnu_E! \bsnu_{E^c}! \bsbeta_E(\bsnu) \bsbeta_{E^c}(\bsnu) 
\;.
\end{array}
$$
We estimate $\bsbeta_E(\bsnu)$ and $\bsbeta_{E^c}(\bsnu)$.
We observe that in case that all $\nu_j \geq 1$ for $j\in E$
$$
{\bsbeta}_E(\bsnu) 
=
\prod_{j=1}^J \phi(\kappa)  
\left(\frac{4\| \bsb \|_{\ell^1}}{\eps}\right)^{\nu_j} 
\leq
\left( \phi(\kappa) \frac{4 \| \bsb \|_{\ell^1}}{\eps} \right)^{|\bsnu_E|} \, ,
$$
which is of product weight form.
In case some or all $\nu_j=0$ for $j\in E$, 
we find the bound
$$
\bsbeta_E(\bsnu)
\leq (\phi(\kappa))^J
\left( \frac{4 \| \bsb \|_{\ell^1}}{\eps} \right)^{|\bsnu_E|}\;,
$$
where we recall that $J = J(\bsb,\eps)$ does not depend on $\bsnu$.

Next we consider $\bsbeta_{E^c}(\bsnu)$. 
Using $\phi(\rho_j) \leq \phi(\kappa)$, we obtain that
$$
\begin{array}{rcl}
\bsbeta_{E^c}(\bsnu) 
& \leq & \displaystyle
\phi(\kappa)^{| \bsnu_{E^c} |}
\prod_{j\in E^c: \nu_j > 0} 
\left(\frac{2b_j}{\eps} \right)^{\nu_j} \left(\frac{1+|\bsnu_{E^c}|}{\nu_j}\right)^{\nu_j} 
\\
& \leq & \displaystyle
\prod_{j\in E^c: \nu_j > 0} 
\left(\frac{2\phi(\kappa) b_j}{\eps} \right)^{\nu_j} 
\left(\frac{1+|\bsnu_{E^c}|}{\nu_j}\right)^{\nu_j}
\;.
\end{array}
$$
We set 
$d_j := 2\phi(\kappa) b_j/\eps$, and $\bar{d}_j := e d_j$ for $j\in E^c$. 
Then 
$$
\bsbeta_{E^c}(\bsnu) 
\leq 
\prod_{j>J} \left(d_j \frac{1+|\bsnu_{E^c}|}{\nu_j}\right)^{\nu_j}
= 
\frac{(1+|\bsnu_{E^c}|)^{|\bsnu_{E^c}|}}{\bsnu_{E^c}^{\bsnu_{E^c}}}
\prod_{j>J} d_j^{\nu_j}
\;.
$$
%
%
Stirling's approximation implies that for all $n \in \bbN$ we have 
$\sqrt{2\pi} n^{n+1/2} \le n! e^n \le e n^{n+1/2}$. 
This also implies that $(1+n)^n \le n^n e^2/\sqrt{2\pi} \le 3 n^n$. 
Thus
$$
\frac{(1+|\bsnu_{E^c}|)^{|\bsnu_{E^c}|}}{\bsnu_{E^c}^{\bsnu_{E^c}}} \le 3 \frac{|\bsnu_{E^c}|^{|\bsnu_{E^c}|}}{\bsnu_{E^c}^{\bsnu_{E^c}}} \le 3 \frac{|\bsnu_{E^c}|! e^{|\bsnu_{E^c}|} }{\sqrt{2\pi} \sqrt{|\bsnu_{E^c}|}} \prod_{\satop{ j \in E^c }{\nu_j > 0} } \frac{e \sqrt{\nu_j}}{\nu_j! e^{\nu_j}} \le \frac{3}{\sqrt{2\pi}} \frac{|\bsnu_{E^c}|!}{\bsnu_{E^c}!} \frac{ \prod_{j \in E^c} e \sqrt{\nu_j} }{\sqrt{|\bsnu_{E^c}| } } \; .
$$
Since $e \sqrt{\nu_j} \le e^{\nu_j}$ for integers $\nu_j > 0$, we obtain
$$
\bsbeta_{E^c}(\bsnu) 
\leq \frac{3}{\sqrt{2\pi}}
\frac {|\bsnu_{E^c}|!}{\bsnu_{E^c}!} \bar \bsd^{\bsnu_{E^c}}  \; .%
$$
%
Combining all bounds, we find 
there exists a constant $\hat{C}>0$ 
(depending on $p$, $\eps$, and on $\bsb$) 
such that there holds, 
for every $\bsnu\in \IN_0^\IN$ with $|\bsnu|<\infty$,
the bound
$$
\begin{array}{rcl}
\displaystyle 
\bsnu_E! \bsnu_{E^c}! \bsbeta_E(\bsnu) \bsbeta_{E^c}(\bsnu)
& \leq & \displaystyle
\frac{3}{\sqrt{2\pi}} \phi(\kappa)^J  
\left(\bsnu_E! \prod_{j=1}^J
\left(\frac{4\| \bsb \|_{\ell^1}}{\eps}\right)^{\nu_j} 
\right)
\times 
| \bsnu_{E^c} |! \prod_{j>J} \bar{d}_j^{\bsnu_j} 
\\
& = & \displaystyle
\hat{C}
\bsnu_E! 
\prod_{j\in E}\beta_j^{\nu_j}  
\times 
|\bsnu_{E^c}|! 
\prod_{j\in {E^c}}\beta_j^{\nu_j} \;.
\end{array}
$$
Here, 
$\beta_j = 4\| \bsb \|_{\ell^1}/\eps$ for $1\leq j \leq J$
is independent of $\bsnu$
and we have $\beta_j = \bar{d}_j \sim b_j$ for $j>J$.
By the choice of $J(\eps,\bsb)$, $\hat{C} = e\phi(\kappa)^{J(\eps,\bsb)}$ 
depends on $\bsb$ and $\eps$, but not on $\bsnu$.
\qed

\begin{remark}\label{remk:ChoicS}
We see from the proof of Theorem \ref{thm:DsiboundC}
and, in particular, from  \eqref{eq:defS},
that the ``crossover-dimension'' $J (\bsb,\eps)$ between product weights
and the more general 
hybrids of product and of SPOD weights,
depends on the precise structure of the 
decay of the sequence $\bsb$
(rather than only on the summability exponent).
It is therefore of some interest to identify cases where $J$ is large. 
This occurs for sequences $\bsb$ which exhibit a 
``plateau'' up to dimension $J>>1$, i.e.
\begin{equation}\label{eq:KLplateau}
b_1 = b_2 = ... = b_J > b_{J + 1} \geq b_{J + 2} \geq ... \downarrow 0 
\;.
\end{equation}
Such cases appear, for example, in \KL expansions of random fields 
$u(\bsy)$, given by \eqref{eq:uviapsi}, 
with two-point correlation kernels which concentrate on a 
(non-dimensional) spatial correlation length scale 
$0 < \lambda << 1$, in $D\subset \IR^d$ a bounded domain.
In this case, typically $J \sim 1/\lambda^d$.
E.g. for $\lambda \sim 0.01$ in three space dimensions,
$J\sim 10^6$.
\end{remark}
To exploit the derivative bounds \eqref{eq:HybBd},
it is of utmost importance to have a fast CBC 
construction of higher-order QMC rules which are able to exploit
\eqref{eq:KLplateau}.
We address a suitable CBC construction of corresponding QMC rules and 
estimates of the QMC errors incurred by these rules
in the ensuing sections, thereby extending \cite{BDGP11,BDLNP12,DKGNS13}.
%
\section{Quasi-Monte Carlo integration}
\label{sec:QMC}
In Theorem \ref{thm:DsiboundC} we established bounds
on the derivatives of $(\bsb,p,\eps)$-analytic solution
families of smooth, nonlinear parametric operator equations
with $(\bsb,p,\eps)$-analytic operators.
Here, we establish error bounds for QMC quadratures 
for these integrand functions. The convergence
estimates obtained here are uniform in the dimension $s$ of the parameter
domain. The application of the QMC quadratures to the 
formally countably-parametric problems must therefore be prepared
by {\em dimension truncation}, i.e. we consider \eqref{eq:mainstrunc}
and its (Petrov-)Galerkin discretization \eqref{eq:parmOpEqh_trun}.
As we explained in the introduction, in order to approximate the
mathematical expectation of the random solutions
by QMC methods, we truncate the infinite sum in
\eqref{eq:uviapsi} to a finite number of $s\geq 1$ terms.
\subsection{Higher-order QMC quadrature based on digital nets}
\label{sec:HoQMCDigNet}
For an integrand $G\in C^0([0,1]^s)$,
we want to approximate the $s$-dimensional integral
\begin{equation}\label{eq:IsF}
 I_s(G) \,:=\,
 \int_{[0,1]^s} G(\bsy) \,\rd\bsy
\end{equation}
by an equal weight QMC quadrature rule of the form
\begin{equation}\label{eq:QNs}
  Q_{N,s}(G) \,:=\,
  \frac{1}{N} \sum_{n=0}^{N-1} G(\bsy_n)\;,
\end{equation}
with judiciously chosen points $\bsy_0,\ldots,\bsy_{N-1} \in [0,1]^s$. 
For completeness we repeat the necessary definitions 
and results from \cite{DKGNS13} in the following.
\begin{definition}[Norm and function space] \label{def_F_norm}
Let $\alpha, s\in\bbN$, $1\le q \le \infty$ and $1\le r \le \infty$, 
and let 
$\bsgamma = (\gamma_\setu)_{\setu\subset\bbN}$ 
be a collection of nonnegative real numbers, known as \emph{weights}. 
Assume further that for every $s\in \IN$, the integrand function
$G: [0,1]^s \to \mathbb{R}$ 
has partial derivatives of orders up to $\alpha$ with respect
to each variable. 
Set $0/0 := 0$ and $a/0 := \infty$ for $a > 0$.
We quantify the smoothness of the integrand function $G$
in \eqref{eq:IsF} 
by the higher order unanchored Sobolev norm
\footnote{We point out that \eqref{eq:defFabs} differs from the 
expression for the norm given in \cite{DKGNS13} which contains
a misprint and which should read as in \eqref{eq:defFabs}.}
\begin{equation}\label{eq:defFabs}
\begin{array}{rl}
\|G\|_{s,\alpha,\bsgamma,q,r}
& \displaystyle
:=
 \Bigg( \sum_{\setu\subseteq\{1:s\}} \Bigg( \gamma_\setu^{-q}
 \sum_{\setv\subseteq\setu} \sum_{\bstau_{\setu\setminus\setv} \in \{1:\alpha\}^{|\setu\setminus\setv|}} \\
&\qquad\qquad\quad
 \displaystyle
 \int_{[0,1]^{|\setv|}} \bigg|\int_{[0,1]^{s-|\setv|}} \!
 (\partial^{(\bsalpha_\setv,\bstau_{\setu\setminus\setv},\bszero)}_\bsy G)(\bsy) \,
\rd \bsy_{\{1:s\} \setminus\setv}
 \bigg|^q \rd \bsy_\setv \Bigg)^{r/q} \Bigg)^{1/r}, 
\end{array}
\end{equation}
with the obvious modifications if $q$ or $r$ is infinite.
Here $\{1:s\}$
is a shorthand notation for the set $\{1,2,\ldots,s\}$, and
$(\bsalpha_\setv,\bstau_{\setu\setminus\setv},\bszero)$ denotes a sequence
$\bsnu$ with $\nu_j = \alpha$ for $j\in\setv$, $\nu_j = \tau_j$ for
$j\in\setu\setminus\setv$, and $\nu_j = 0$ for $j\notin\setu$.
Let $\calW_{s,\alpha,\bsgamma,q,r}$ denote the Banach space of all such
functions $F$ with finite norm.
\end{definition}

By the definition of $0/0$ and $a/0$,  if $\gamma_\setu = 0$ for some $\setu$ 
then the corresponding term 
$ \sum_{\setv\subseteq\setu} \sum_{\bstau_{\setu\setminus\setv} \in \{1:\alpha\}^{|\setu\setminus\setv|}}
    \int_{[0,1]^{s-|\setv|}}
   (\partial^{(\bsalpha_\setv,\bstau_{\setu\setminus\setv},\bszero)}_\bsy G)(\bsy) \,\rd \bsy_{\{1:s\} 
   \setminus\setv} $ 
has to be $0$ for all $G \in \calW_{s, \alpha, \bsgamma, q, r}$.

The following result is an upper bound on the worst-case integration error 
in $\calW_{s, \alpha,\bsgamma, q, r}$ using a QMC rule based on a digital net, 
see \cite[Theorem~3.5]{DKGNS13}.

\begin{theorem}[Worst case error bound] \label{thm:wce}
Let $\alpha, s\in\bbN$ with $\alpha>1$, $1\le q\le \infty$ and $1\le r\le \infty$, and let
$\bsgamma = (\gamma_\setu)_{\setu\subset\bbN}$ denote a collection of
weights. Let $r'\ge 1$ satisfy 
$1/r + 1/r' = 1$. Let $b$ be
prime, $m\in\bbN$, and let $\calS=\{\bsy_n\}_{n=0}^{b^m-1}$ denote a
digital net with generating matrices $C_1,\ldots,C_s\in\bbZ_b^{\alpha
m\times m}$. Then we have
\[
  \sup_{\|G\|_{s,\alpha,\bsgamma,q,r} \le 1}
  \left| \frac{1}{b^m} \sum_{n=0}^{b^m-1} G(\bsy_n) - \int_{[0,1]^s} G(\bsy) \,\mathrm{d} \bsy \right|
  \,\le\, e_{s,\alpha,\bsgamma,r'}(\calS)\;,
\]
with
\begin{align} \label{def-B}
  e_{s,\alpha,\bsgamma,r'}(\calS)
  \,:=\, \Bigg(\sum_{\emptyset \neq \setu \subseteq \{1:s\}}
  \bigg(C_{\alpha,b}^{|\setu|}\, \gamma_\setu \sum_{\bsk_\setu \in \setD_\setu^*}
  b^{-\mu_{\alpha}(\bsk_\setu)} \bigg)^{r'} \Bigg)^{1/r'}\;.
      \end{align}
Here $\setD_\setu^*$ is the ``dual net without $0$ components'' projected
to the components in $\setu$, defined by
\begin{equation} \label{dual}
  \setD_\setu^* \,:=\, \left\{ \bsk_\setu \in \bbN^{|\setu|}\,:\,
  \sum_{j\in\setu} C_j^\top {\rm tr}_{\alpha m}(k_j) = \bszero \in \bbZ_b^m \right\} \;,
\end{equation}
where ${\rm tr}_{\alpha m}(k) := (\varkappa_0, \varkappa_1, \ldots,
\varkappa_{\alpha m-1})^\top$ if $k = \varkappa_0 + \varkappa_1 b
+\varkappa_2 b^2 + \cdots$ with $\varkappa_i\in \{0,\ldots,b-1\}$.
Moreover, we have $\mu_{\alpha}(\bsk_\setu) = \sum_{j\in\setu}
\mu_\alpha(k_j)$ with
\begin{equation} \label{def-mu-k}
  \mu_\alpha(k)
  \,:=\,
  \begin{cases}
  0 & \mbox{\,if } k = 0, \\
  a_1 + \cdots + a_{\min(\alpha,\rho)} & 
  \begin{aligned}
   \mbox{if }
  k &= \kappa_1 b^{a_1-1} + \cdots + \kappa_\rho b^{a_\rho-1} \mbox{ with} \\
    &\kappa_i\in \{1,\ldots,b-1\} \mbox{ and } a_1>\cdots>a_\rho>0,
  \end{aligned}
  \end{cases}
\end{equation}
and
\begin{align}\label{eq:Cab}
 &C_{\alpha,b}
 \,:=\,
 \max\left(\frac{2}{(2\sin\frac{\pi}{b})^{\alpha}},\max_{1\le z\le\alpha-1}
 \frac{1}{(2\sin\frac{\pi}{b})^z}\right)
 \nonumber
 \\
 &\qquad\qquad\qquad\times
 \left(1+\frac{1}{b}+\frac{1}{b(b+1)}\right)^{\alpha-2}
 \left(3 + \frac{2}{b} + \frac{2b+1}{b-1} \right)\;.
\end{align}
\end{theorem}
\begin{remark}\label{rem_Y}
For the special but important case $b=2$, Yoshiki~\cite{Y15} 
achieved an improvement of the constant $C_{\alpha,2}$. 
He showed that one can choose $C_{\alpha, 2} = 2^{-1/q'}$, where 
$1 \le q' \le \infty$ is the H\"older conjugate of $q$, 
i.e. $1/q + 1/q'=1$, and $q$ is the parameter appearing 
in the norm \eqref{eq:defFabs}.
\end{remark}
We recall the special case where the integrand
$G(\bsy)$ is a composition of a continuous, linear functional $\LinFun(\cdot) \in \bcX'$
with the (Petrov-)Galerkin approximation $q^s_h(2 \bsy-\bsone)$ 
of the dimension-truncated, parametric and
$(\bsb,p,\eps)$-holomorphic, operator equation \eqref{eq:NonOpEqn}.  
In this case, 
for every $s\in N$ and for every $h>0$ sufficiently small,
the integrand functions
$G(\bsy) := (\LinFun \circ q^s_h)(\bsy_{\{1:s\}})$ are, likewise, 
$(\bsb,p,\eps)$-holomorphic {\em uniformly w.r.t. $s\in \IN$ and to $h>0$}.
By Theorem \ref{thm:DsiboundC}, they 
satisfy the derivative estimates \eqref{eq:HybBd}
uniformly w.r.t. $s\in \IN$ and to $h>0$.
For integrand functions $G(\bsy)$ which satisfy \eqref{eq:HybBd}, 
we proved in \cite{DKGNS13} convergence rates of QMC quadratures
which are based on higher order digital nets. Precisely, 
we showed in \cite[Section~3]{DKGNS13} a special case of the following result.
\begin{proposition}\label{prop:main1}
Let $s\ge 1$ and $N = b^m$ for $m\ge 1$ and prime $b$. 
Let $\bsbeta = (\beta_j)_{j\ge 1}$ be a sequence of positive numbers, 
and denote by $\bsbeta_s = (\beta_j)_{1\le j \le s}$ its $s$-term
truncation. 
Assume that
\begin{equation} \label{p-sum}
  \exists\, 0<p\le 1 : \quad \sum_{j=1}^\infty \beta_j^p < \infty\;.
\end{equation}
Define, for $0<p<1$ as in \eqref{p-sum},
\begin{equation} \label{alpha}
  \alpha \,:=\, \lfloor 1/p \rfloor +1 \;.
\end{equation}
Consider integrand functions $G(\bsy)$ 
whose mixed partial derivatives 
of order $\alpha$ satisfy
\begin{equation} \label{eq:like-norm}
\forall\, \bsy\in U \;\forall s\in \IN \;
\forall\, \bsnu \in \{0, 1, \ldots, \alpha\}^s:
\quad
 | (\partial^{\bsnu}_\bsy G)(\bsy)| 
\,\le\, 
c(G) \bsnu_E! 
\prod_{j\in E}\beta_j^{\nu_j}  
\times 
|\bsnu_{E^c}|! 
\prod_{j\in {E^c}}\beta_j^{\nu_j}
\end{equation}
for some fixed integer $J\in \mathbb{N}$ 
where $E = \{1, 2, \ldots, J\}$ and ${E^c} = \mathbb{N} \setminus E$,
and where $c(G)>0$ is independent of $\bsy$, $s$ and of $\bsnu$.
Then, for every $N\in \N$, 
an interlaced polynomial lattice rule of
order $\alpha$ with $N$ points can be constructed using a fast
component-by-component algorithm, using
$\calO(\alpha \left(\min\{s, J\} + \alpha (s-J)_+ \right) N \log N)$ 
operations, plus $\calO(\alpha^2 (s-J)_+^2 N)$ update cost, plus 
$\calO(N + \alpha (s- J)_+ N)$ memory cost, where $(w)_+ = \max\{0, w\}$, 
such that there holds the error bound
\begin{equation}\label{thebound}
\forall s,N \in \N:\quad 
  |I_s(G) - Q_{N,s}(G)|
  \,\le\, C_{\alpha,\bsbeta,b,p}\, N^{-1/p} \;,
\end{equation}
where $C_{\alpha,\bsbeta,b,p} < \infty$ is a constant independent of $s$
and $N$.
\end{proposition}
\begin{proof}

For a function $G$ satisfying \eqref{eq:like-norm}, 
its norm \eqref{eq:defFabs} with $r=\infty$ and for any $q$, 
can be bounded by
\begin{align*}
 \|G\|_{s,\alpha,\bsgamma,q,\infty}
 &\,\le\, c
 \max_{\setu\subseteq\{1:s\}}
 \gamma_\setu^{-1}
 \sum_{\bsnu_\setu \in \{1:\alpha\}^{|\setu|}}
 \bsnu_{\setu \cap E}!\,
 \prod_{j\in\setu\cap E} \left(2^{\delta(\nu_j,\alpha)}\beta_j^{\nu_j}\right)\;
 |\bsnu_{\setu \cap {E^c}}|!\,
 \prod_{j\in\setu \cap {E^c}} \left(2^{\delta(\nu_j,\alpha)}\beta_j^{\nu_j}\right)\;\\
 &\,=\, c(G)
  \max_{\setu\subseteq\{1:s\}}
 \gamma_\setu^{-1}
 \sum_{\bsnu_\setu \in \{1:\alpha\}^{|\setu|}}
 \bsnu_{\setu \cap E}!\,
 |\bsnu_{\setu \cap {E^c}}|!\,
 \prod_{j\in\setu} \left(2^{\delta(\nu_j,\alpha)}\beta_j^{\nu_j}\right)\;,
\end{align*}
where $\delta(\nu_j,\alpha)$ is $1$ if $\nu_j=\alpha$ and is $0$
otherwise. To make  $\|G\|_{s,\alpha,\bsgamma,q,\infty} \le c$, 
we choose
\begin{equation}\label{equ:hybridWeight}
 \gamma_\setu :=  \sum_{\bsnu_\setu \in \{1:\alpha\}^{|\setu|}}
 \bsnu_{\setu \cap E}!\,
 |\bsnu_{\setu \cap {E^c}}|!\,
 \prod_{j\in\setu} \left(2^{\delta(\nu_j,\alpha)}\beta_j^{\nu_j}\right).
\end{equation}

With that, we can apply \cite[Theorem 5.3]{DKGNS13} to get the estimate \eqref{thebound}.
We remark that when $E = \emptyset$, we recover the case of SPOD weights 
as in equation (3.17) in \cite{DKGNS13}.
\end{proof}
\begin{remark}
A more precise bound \eqref{thebound} with an explicit constant $C_{\alpha,\bsbeta, b, p}$ 
is given in Eq.~\eqref{bound_with_constant} below.
\end{remark}
\begin{remark}\label{rem_interval_change}
Notice that the bound \eqref{eq:like-norm} was shown in Theorem~\ref{thm:DsiboundC} 
for functions defined on $[-1,1]^{\mathbb{N}}$, 
whereas now we use (the dimension truncated version) $[0,1]^s$. 
The change from $[-1,1]$ to $[0,1]$ can be achieved by the simple 
linear transformation $y \mapsto (y+1)/2$. 
Using \eqref{eq:HybBd} together with this change of variable in Proposition~\ref{prop:main1} 
increases the constant in \eqref{eq:Cab} by a factor of at most $2^\alpha$. 
Thus, in order for the theory to apply to the integrands from 
Sections~\ref{sec:HolOpEq} and \ref{sec:anadepsol}, 
we need to multiply $C_{\alpha, b}$ in \eqref{eq:Cab} by $2^\alpha$. 
In other words we need to replace 
$C_{\alpha, b}$ by $C'_{\alpha, b} = 2^\alpha C_{\alpha, b}$.
\end{remark}

\subsection{Combined error bound}
\label{sec:CombErrBd}
From the error bound in Theorem \ref{thm:trunc} 
on the impact of dimension truncation, the QMC integration error bound in
Proposition \ref{prop:main1}, 
and from the properties \eqref{eq:quasiopt} and \eqref{eq:convrate}
    of the (Petrov-)Galerkin projection \eqref{eq:parmOpEqh_trun}
we obtain
\begin{theorem} \label{thm:Combin}
Consider the nonlinear, parametric operator equation 
\eqref{eq:main} under the assumptions made 
in Section \ref{sec:OpEqUncInp},
and under Assumption \ref{asmp:psum}
on $p$-summability \eqref{eq:psumpsi} and 
the decreasing arrangement \eqref{eq:ordered}
of the sequence $\bsb$.
If the approximation property \eqref{eq:apprprop} 
holds, and if the parametric 
solutions $q(\bsy)$ 
of the problems \eqref{eq:paraOpEq}
are uniformly $\bcX_t$-regular
in the sense that 
there exists $C(F,t)<\infty$ such that
\be\label{eq:qRegt}
\sup_{\bsy\in U} \| q(\bsy) \|_{\bcX_t} \leq C(F,t) <\infty 
\;,
\ee
then for the QMC-integrated, (Petrov-)Galerkin-approximated 
responses $Q_{N,s}(\LinFun(q^s_h))$ of the parametric 
(Petrov-)Galerkin approximations $q^s_h(\bsy) \in \bcX^h$ 
defined in \eqref{eq:parmOpEqh_trun}, 
there holds the error bound
$$
\left| I(\LinFun(q(\cdot))) - Q_{N,s}(\LinFun(q^s_h)) \right|
\leq 
C_1
(N^{-1/p} + h^t + s^{-(1/p-1)})
\;.
$$
Here, the constant $C_1>0$ is independent of $N$, $h$ and of $s$.
\end{theorem}
\begin{proof}
We write
$$
\begin{array}{rcl}
\left| I(\LinFun(q)) - Q_{N,s}(\LinFun(q^s_h)) \right|
&\leq& \displaystyle 
\left| I(\LinFun(q)) - I_s(\LinFun(q^s)) \right|
+
\left| I_s(\LinFun(q^s)) - Q_{N,s}(\LinFun(q^s)) \right| 
\\
& & \displaystyle
+
\left| Q_{N,s}(\LinFun(q^s - q^s_h)) \right|
\\
& =: & \displaystyle 
E_I + E_{II} + E_{III}
\;.
\end{array}
$$
The dimension truncation error $E_I$ 
is bounded by \eqref{eq:Idimtrunc} and \eqref{eq:DTbound} 
in Theorem \ref{thm:trunc}. 
Term $E_{II}$ is a QMC error which is bounded 
by Proposition \ref{prop:main1}; this Proposition is applicable  
based on Theorem \ref{thm:DsiboundC}, upon noting \eqref{eq:y1s}, i.e., 
that for finite truncation dimension $s$ the dimensionally truncated,
parametric solution $q^s(\bsy)$ can be interpreted as evaluation of
$q(\bsy)$ (to which Theorem \ref{thm:DsiboundC} applies) 
at the particular parameter value $\bsy := (\bsy_{\{1:s\}}, \bszero)$.
The last term $E_{III}$ is bounded using the equal weight property 
\eqref{eq:QNs} of $Q_{N,s}$ to infer
$$
E_{III}
\leq \| \LinFun \|_{\cX'} 
\sup_{\bsy_{\{1:s\}} \in [-1/2,1/2)^s} 
\| q^s(\bsy_{\{1:s\}}) - q^s_h(\bsy_{\{1:s\}}) \|_{\cX}
\leq 
\| \LinFun \|_{\cX'} 
\sup_{\bsy\in U} \| q(\bsy) - q_h(\bsy) \|_{\cX}
\;
$$
and the (Petrov-)Galerkin error 
$\sup_{\bsy \in U} \|q(\bsy) - q_h(\bsy)\|_{\cX} \le C h^t$.
\end{proof}
\section{Fast component-by-component construction}
\label{sec:HybrCBC}
Here, we outline, based on \cite{NC06a,DiGo12,DKGNS13}, a modification
of the fast CBC construction of the generating vector for the QMC
rule; while asymptotically, as $s\to\infty$, the complexity of this
construction equals that of the CBC construction 
for the SPOD weights in \cite{DKGNS13},
for finite, large values of the index $J$ in the proof of 
Theorem \ref{thm:DsiboundC} (which do occur in practical
situations as outlined in Remark \ref{remk:ChoicS}), 
we obtain quantitative advantages for the construction 
based on `` hybrid QMC-weights '', as outlined 
in what follows. 
We follow \cite{DKGNS13} closely in our exposition below.

As quadrature rule we use (interlaced) polynomial lattice
rules which are a special class of (higher order) digital nets, 
and which were introduced by Niederreiter, see \cite{Nie92}, see also \cite{DiPi10, LP14, Nuy13}.
We state the definition of these rules in the following. 
Let $b$ be a prime number, $\Z_b$ be the finite field with $b$ elements, 
$\Z_b[x]$ be the set of all polynomials with coefficients in $\Z_b$ and $\Z_b((x^{-1}))$ 
be the set of all formal Laurent series $\sum_{\ell = w}^\infty t_\ell x^{-\ell}$, 
where $w$ is an arbitrary integer and $t_\ell \in \Z_b$ for all $\ell$.
\begin{definition}[Polynomial lattice rules] \label{def_poly_lat}
For a prime $b$ and any $m \in \bbN$, let $P \in \Z_b[x]$ be an
irreducible polynomial with $\deg(P)= m$. 
For a given dimension $s\geq 1$, select $s$ polynomials
$q_1(x),\ldots,q_s(x)$ from the set
\begin{equation}\label{eq:Gbm}
   \Gc_{b, m} \,:=\, \{q(x) \in \Z_b[x] \setminus \{0\} \,:\, \text{deg}(q) < m\}\;,
\end{equation}
and write collectively
\begin{equation}\label{eq:GenVecq}
 \bsq \,=\, \bsq(x) \,=\, (q_1(x),\ldots,q_s(x)) \in
\Gc^s_{b, m}  \;.
\end{equation}
%
For each integer $0 \le n < b^m$, 
let $n = \eta_0 + \eta_1 b + \cdots + \eta_{m-1} b^{m-1}$ be the
$b$-adic expansion of $n$, and associate with $n$ the polynomial
\[
  n(x) = \sum_{r=0}^{m-1} \eta_r \, x^r  \in \Z_b[x]\;.
\]
Furthermore, we denote by $v_{m}$ the map from $\Z_b((x^{-1}))$ to the
interval $[0,1)$ defined for any integer $w$ by
\[
  v_{m}\left( \sum_{\ell=w}^\infty t_\ell\, x^{-\ell} \right) =
  \sum_{\ell=\max(1,w)}^m t_\ell \, b^{-\ell}\;.
\]
Then, the QMC point set $\calS_{P,b,m,s}(\bsq)$ of a (classical)
\emph{polynomial lattice rule} comprises the points
\[
  \bsy_n = \left( v_{m} \left( \frac{n(x) q_1(x)}{P(x)} \right),\ldots,
    v_{m} \left( \frac{n(x) q_s(x)}{P(x)} \right) \right) \in [0,1)^s,
    \quad n = 0, \ldots, b^m-1\;.
\]
\end{definition}

Interlaced polynomial lattice rules are special families of higher order digital nets \cite{D07, D08}. These quadrature rules were first studied in \cite{DiGo12, Go13, Go13b} since they yield faster CBC constructions.

\begin{definition}[Interlaced polynomial lattice rules] \label{int-poly}
Define the \emph{digit interlacing function}
with interlacing factor $\alpha \in \mathbb{N}$
by 
\begin{equation}\label{eq:DigIntl}
\begin{array}{rcl}
\mathscr{D}_\alpha: [0,1)^{\alpha} & \to & [0,1)
\\
(x_1,\ldots, x_{\alpha}) &\mapsto & \sum_{a=1}^\infty \sum_{j=1}^\alpha
\xi_{j,a} b^{-j - (a-1) \alpha}\;,
\end{array}
\end{equation}
where $x_j = \xi_{j,1} b^{-1} + \xi_{j,2} b^{-2} + \cdots$ for $1 \le j
\le \alpha$. We also define such a function for vectors by setting
\begin{equation} \label{eq:DigIntlAr}
\begin{array}{rcl}
\mathscr{D}_\alpha: [0,1)^{\alpha s} & \to & [0,1)^s
\\
(x_1,\ldots, x_{\alpha s})
&\mapsto &
(\mathscr{D}_\alpha(x_1,\ldots, x_\alpha),  \ldots,
\mathscr{D}_\alpha(x_{(s-1)\alpha +1},\ldots, x_{s \alpha}))
\;.
\end{array}
\end{equation}
Then, an \emph{interlaced polynomial lattice rule of order $\alpha$ with
$b^m$ points in $s$ dimensions} is a QMC rule using
$\mathscr{D}_\alpha(\calS_{P,b,m,\alpha s}(\bsq)) 
= \{ \mathscr{D}_{\alpha}(\bsy_n): n = 0, \ldots, b^m-1 \}$ 
as quadrature points,
for some given modulus $P$ and generating vector
$\bsq \in \Gc^{\alpha s}_{b, m}$.
\end{definition}


We have the following upper bound for the worst-case error
of interlaced polynomial lattice rules \cite[Section 3.2]{DKGNS13} 
\begin{equation} \label{eq:error-2}
  e_{s,\alpha,\bsgamma,1}(\calS)
  \,\le\,
  \sum_{\emptyset\neq \setv \subseteq\{1:\alpha s\}}
  (C'_{\alpha,b})^{|\setu(\setv)|}\, \gamma_{\setu(\setv)}\,b^{\alpha(\alpha-1)|\setu(\setv)|/2}
  \sum_{\bsell_\setv \in \calD_\setv^*}b^{-\alpha\mu_1(\bsell_\setv)}
  \;,
\end{equation}
where $\calD_\setv^*$ is the ``dual net without $0$ components''
defined in terms of the generating polynomials, 
see \cite[Eq. (3.28)]{DKGNS13} and where we replaced 
$C_{\alpha, b}$ by $C'_{\alpha, b}$.
Eq.~\eqref{eq:error-2} is derived from \eqref{def-B} 
by setting $r'=1$ and using interlaced polynomial lattice rules, 
see \cite{DKGNS13} for details.
Here, 
for a given set $\emptyset\ne\setv\subseteq\{1:\alpha s\}$, we define
\begin{equation} \label{u_of_v}
  \setu(\setv) \,:=\, \{\lceil j/\alpha \rceil: j \in \setv\} \,\subseteq\, \{1:s\}\;,
\end{equation}
where each element appears only once. The set
$\setu(\setv)$ can be viewed as an indicator on whether the set $\setv$
includes any element from each block of $\alpha$ components from
$\{1:\alpha s\}$.

Since we do not have a suitable expression for the worst-case error 
$e_{s, \alpha, \bsgamma, 1}$ we use the right-hand side of \eqref{eq:error-2} 
as our search criterion in the CBC construction instead. 
To simplify our notation, we define
\begin{equation} \label{eq:def-E}
  \cE_d(\bsq) \,:=\,
  \sum_{\emptyset \neq \setv \subseteq \{1:d\}} \widetilde{\gamma}_{\setv}
  \sum_{\bsell_\setv \in \calD_\setv^*}
  b^{- \alpha \mu_1(\bsell_\setv)}
\;.
\end{equation}
The case $d = \alpha s$ and the weights
\begin{equation} \label{eq:weight2}
  \widetilde{\gamma}_{\setv} \,:=\,
  (C'_{\alpha,b})^{|\setu(\setv)|}\, \gamma_{\setu(\setv)}\,b^{\alpha(\alpha-1)|\setu(\setv)|/2}\;
\end{equation}
are of particular interest for our purposes here. However, as shown in \cite{DKGNS13}, the theorem below holds for any $d$ and also for general weights $\widetilde{\gamma}_{\setv}$.

\begin{theorem}[CBC error bound]\label{thm_cbc}
Let $b \ge 2$ be prime, and $\alpha \ge 2$ and $m,d \ge 1$ be integers,
and let $P \in \Z_b[x]$ be an irreducible polynomial with $\deg(P) = m$.
Let $(\widetilde{\gamma}_{\setv})_{\setv\subseteq\{1:d\}}$ be positive
real numbers. Then a generating vector 
$\bsq^* = (1, q_2^*,
\ldots, q_d^*) \in \Gc^d_{b,m}$ can 
be constructed using a component-by-component
approach, minimizing $\cE_d(\bsq)$ in each step, such that
\begin{equation} \label{cond1}
 \cE_d(\bsq^*) \,\le\,
 \Bigg( \frac{2}{b^{m}-1}\sum_{\emptyset\ne\setv\subseteq\{1:d\}}
         \widetilde\gamma_{\setv}^{\lambda} \left(\frac{b-1}{b^{\alpha\lambda}-b}\right)^{|\setv|}
 \Bigg)^{1/\lambda}
 \quad\mbox{for all}\quad \lambda \in (1/\alpha, 1]
 \;.
\end{equation}
\end{theorem}

It follows from Theorem~\ref{thm_cbc} that an interlaced polynomial lattice rule
with interlacing factor $\alpha$ in $s$ dimensions can be constructed using a CBC algorithm with weights \eqref{eq:weight2}, such that
\begin{align*}
  &e_{\alpha,\bsgamma,s,1}(\calS)
  \,\le\, \cE_{\alpha s}(\bsq^*) \\
  &\,\le\, \Bigg( \frac{2}{b^{m}-1}
  \sum_{\emptyset\ne\setv\subseteq\{1:\alpha s\}}
  \left((C'_{\alpha,b})^{|\setu(\setv)|}\, \gamma_{\setu(\setv)}\, b^{\alpha(\alpha-1)|\setu(\setv)|/2}\right)^\lambda
                                                                \left(\frac{b-1}{b^{\alpha\lambda}-b} \right)^{|\setv|}
  \Bigg)^{1/\lambda} \\
  &\,=\, \Bigg( \frac{2}{b^{m}-1}
  \sum_{\emptyset\ne\setu\subseteq\{1:s\}}
  \left( (C'_{\alpha,b})^{|\setu|}\, \gamma_{\setu}\, b^{\alpha(\alpha-1)|\setu|/2}\right)^\lambda
  \left(\left(1+\frac{b-1}{b^{\alpha\lambda}-b} \right)^\alpha-1\right)^{|\setu|}
  \Bigg)^{1/\lambda}\;. 
\end{align*}

By substituting in $\gamma_\setu$ from \eqref{equ:hybridWeight} and using Jensen's inequality, we get
\begin{align} \label{eq:want}
  e_{\alpha,\bsgamma,s,1}(\calS)
  &\,\le\, \Bigg( \frac{2}{b^{m}-1}
  \sum_{\emptyset\ne\setu\subseteq\{1:s\}}
  \sum_{\bsnu_\setu \in \{1:\alpha\}^{|\setu|}}
  (\bsnu_{\setu\cap E}! |\bsnu_{\setu\cap {E^c}}|!)^\lambda \,\prod_{j\in\setu}
  \left(B\,2^{\delta(\nu_j,\alpha)} \beta_j^{\nu_j}\right)^\lambda
  \Bigg)^{1/\lambda} \nonumber\\
  &\,=\, \Bigg( \frac{2}{b^{m}-1}
  \sum_{\bszero\ne\bsnu\in\{0:\alpha\}^s}
  (\bsnu_{E \cap \{1:s\}} ! |\bsnu_{{E^c}\cap\{1:s\}}|!)^\lambda \,\prod_{\satop{j=1}{\nu_j>0}}^s
  \left(B\,2^{\delta(\nu_j,\alpha)} \beta_j^{\nu_j}\right)^\lambda
  \Bigg)^{1/\lambda}\;,
\end{align}
where
\begin{align} \label{eq:B}
  B \,:=\, C'_{\alpha,b}\, b^{\alpha(\alpha-1)/2}
  \left(\left(1+\frac{b-1}{b^{\alpha\lambda}-b} \right)^\alpha-1\right)^{1/\lambda}\;.
\end{align}
We now show how we can choose $\lambda$ such that the sum in \eqref{eq:want} is
bounded independently of $s$. Let $\widetilde\beta_j :=
2\max(B,1)\beta_j$. Using the same argument as in \cite{DKGNS13}, the sum in \eqref{eq:want} 
is bounded by
\[
  \sum_{\bszero\ne\bsnu\in\{0:\alpha\}^s}
  \bigg(\bsnu_{E\cap\{1:s\}}!|\bsnu_{{E^c}\cap\{1:s\}}|! \,\prod_{j=1}^s
  \widetilde\beta_j^{\nu_j} \bigg)^\lambda\;,
\]
where each term in the sum to be raised to the power of $\lambda$ is of
the form
\begin{equation} \label{eq:form}
  \nu_1! \nu_2! \cdots \nu_J!
  (\nu_{J+1} + \nu_{J+2} + \cdots + \nu_s)!\,
  \underbrace{\widetilde\beta_1 \cdots \widetilde\beta_1}_{\nu_1}
  \underbrace{\widetilde\beta_2 \cdots \widetilde\beta_2}_{\nu_2}
  \cdots
  \underbrace{\widetilde\beta_s \cdots \widetilde\beta_s}_{\nu_s}\;,
\end{equation}
where for $s\le J$ we set $(\nu_{J+1}+ \nu_{J+2} + \ldots + \nu_s)!= 0! = 1$ and
$\nu_{s+1}! = \cdots = \nu_J! = 0! = 1$.

We now define a sequence $d_j := \widetilde\beta_{\lceil j/\alpha\rceil}$
so that $d_1 = \cdots = d_\alpha = \widetilde\beta_1$ and $d_{\alpha+1} =
\cdots = d_{2\alpha} = \widetilde\beta_2$, and so on. 
Then any term of the
form \eqref{eq:form} is bounded by a term of the form 
\[
\left( \prod_{j \in \setv \cap \alpha E}\alpha! d_j \right)  |\setv\cap \alpha {E^c}|!\, 
\prod_{j\in\setv \cap \alpha {E^c}} d_j
\]
for some finite subset of indices $\setv\subset\bbN$. 
As before, $E = \{1:J\}$ and we write
\[
\alpha E = \{1,2,\ldots,\alpha J \} \text{  and  }
\alpha {E^c} = \{\alpha J+1, \alpha J +2, \ldots, \}
\;.
\]
Thus we conclude that
\begin{align} \label{eq:last}
  \sum_{\bszero\ne\bsnu\in\{0:\alpha\}^s}
  \bigg(\bsnu_{E\cap\{1:s\}} ! \, |\bsnu_{{E^c}\cap\{1:s\}}|! \,\prod_{\satop{j=1}{\nu_j>0}}^s
  \widetilde\beta_j^{\nu_j} \bigg)^\lambda
  &\,{\leq} \, \sum_{\satop{\setv\subset\bbN}{|\setv|<\infty}}
  \bigg( \bigg( \prod_{j \in \setv \cap \alpha E}\alpha! d_j \bigg) |\setv\cap\alpha {E^c}|!\, \prod_{j\in\setv \cap \alpha {E^c}} d_j\bigg)^\lambda \nonumber\\
  &\,=\, \sum_{\setw \subset \alpha E} \sum_{\satop{\setu \subset \alpha {E^c}}{|\setu|<\infty}}
   \left( \left( \prod_{j \in \setw} \alpha! d_j \right)
   |\setu|! \prod_{j \in \setu} d_j\right)^{\lambda} \nonumber \\
  &\,=\,
  \sum_{\setw \subset \alpha E} 
  \left(\prod_{j \in \setw} \alpha! d_j \right)^{\lambda}
  \sum_{\satop{\setu \subset \alpha {E^c}}{|\setu| < \infty}}
  \left(|\setu|! \prod_{j \in \setu} d_j\right)^{\lambda}\nonumber
  \\
  &\,\le\, \prod_{j=1}^{\alpha J} (1 + (\alpha! d_j)^{\lambda})
  \sum_{\ell=0}^\infty (\ell!)^\lambda
  \sum_{\satop{\setu\subset\alpha {E^c}}{|\setu|=\ell}} \prod_{j\in\setu} d_j^\lambda
  \nonumber\\
  &\,\le\, \exp\left( (\alpha!)^\lambda 
        \sum_{j=1}^\infty d_j^{\lambda}\right)      
        \sum_{\ell=0}^\infty (\ell!)^{\lambda-1}
  \bigg(\sum_{j=1}^\infty d_j^\lambda\bigg)^\ell\;,
\end{align}
where in the last step we used the estimation $1 + x = \exp(\log (1+x)) \le \exp(x)$.

Note that $\sum_{j=1}^\infty \beta_j^p < \infty$ holds if and only if
$\sum_{j=1}^\infty d_j^p < \infty$. The last expression
in \eqref{eq:last} is finite for $p \le \lambda\le 1$.
The last expression in \eqref{eq:last} is also finite if $\lambda=1$ and $\sum_{j=1}^\infty d_j < 1$.
Since $\lambda$ also needs to satisfy $1/\alpha < \lambda\le 1$, we choose
\begin{equation}\label{lambda}
  \lambda \,=\, p \qquad\mbox{and}\qquad \alpha \,=\, \lfloor 1/p\rfloor +1\;,
\end{equation}
and for 
$p=1$ we assume additionally that $\sum_{j=1}^\infty d_j
< 1$,
which is equivalent to
\begin{equation} \label{eq:small}
  \sum_{j=1}^\infty \beta_j
  < \frac{1}{2\alpha\max(B,1)} \; .
\end{equation}
Thus with \eqref{lambda}, in \eqref{cond1}
we obtain a convergence of $\calO(N^{-1/p})$
 where $N=b^m$ with the implied constant 
bounded by \eqref{eq:last}, which is
independent of the dimension $s$. 
More precisely, we obtain that the integration error is bounded by (using $N = b^m$)
\begin{equation}\label{bound_with_constant}
\left( \frac{2}{N-1} \right)^{1/p} \left[ \exp\left( ((\lfloor 1/p \rfloor+1)!)^{p} \sum_{j=1}^\infty d_j^p \right) \sum_{\ell=0}^\infty (\ell!)^{p-1} \left( \sum_{j=1}^\infty d_j^p \right)^\ell \right]^{1/p},
\end{equation}
where $d_j = 2 \max(B, 1) \beta_{\lceil j/\alpha \rceil}$ and $B$ is given by \eqref{eq:B}. 
We have not tried to optimize the constant in \eqref{bound_with_constant} in terms of its dependence on $1/p$. Indeed, the expression in brackets in \eqref{bound_with_constant} grows at least of order $\exp(a^{1/p} )$, for some $a > 1$. (If one was mainly interested in the dependence of the constant on $1/p$, then \eqref{cond1} yields a dependence of order $\exp(c/p^2)$ for some $c > 0$.)

\subsection{Component-by-component algorithm} 
\label{sec:CBC}

A version of the component-by-component (or CBC) algorithm was first
proposed by Korobov \cite{Kor59} and rediscovered in \cite{SloRez02} in the context
of lattice rules for periodic functions. A version for deterministically shifted 
lattice rules in weighted spaces was proposed by \cite{SloKuoJoe02},
and the version for randomly shifted lattice rules was proposed in \cite{SloKuoJoe02a}.
Here we focus on the CBC algorithm for higher order interlaced polynomial
lattice rule as proposed in \cite{DiGo12, Go13, Go13b, DKGNS13}. 

We first derive a closed form expression for 
$\cE_d(\bsq)$ in \eqref{eq:def-E} which can be used for computation. 
Recall from Definition~\ref{def_poly_lat} that the
$j$-th coordinate of the $n$-th point of the interlaced polynomial lattice
point set is
\begin{equation*}
  y_j^{(n)} \,=\, \upsilon_m\left(\frac{n(x)\, q_j(x)}{P(x)} \right)\;.
\end{equation*}
Note that $y_j^{(n)}$ depends on the $j$-th component $q_j$ of the generating
vector. In the following we use results from \cite{DiPi05}. We have
\begin{equation*}
 \sum_{\bsell_\setv \in \calD_\setv^*} b^{- \alpha \mu_1(\bsell_\setv)}
 \,=\, \frac{1}{b^m} \sum_{n=0}^{b^m-1} \prod_{j \in \setv} \omega(y_j^{(n)})\;,
\end{equation*}
where $\bsy_{\setv}^{(n)} = (y_j^{(n)})_{j \in \setv}$ is the projection
of the $n$-th point $\bsy^{(n)}$ onto the coordinates in $\setv$, 
\begin{align*}
  \omega(y)
  \,=\, \frac{b-1}{b^\alpha-b} - b^{\lfloor \log_b y \rfloor (\alpha-1)}
  \frac{b^\alpha-1}{b^\alpha - b}\;,
\end{align*}
and where for $y=0$ we set $b^{\lfloor \log_b 0\rfloor (\alpha-1)} :=
0$. The last equality can be obtained by multiplying
\cite[Eq.~(2)]{DiPi05} by $b^{-\alpha}$.
Thus we have
\begin{equation} \label{eq:E-CBC}
  \cE_d(\bsq) \,=\,
  \frac{1}{b^m} \sum_{n=0}^{b^m-1} \sum_{\emptyset \neq \setv \subseteq \{1:d\}}
  \widetilde\gamma_\setv \prod_{j \in \setv} \omega(y_j^{(n)})\;.
\end{equation}

The CBC construction proceeds inductively on the dimension, keeping the components already calculated fixed and searching for the polynomial $q_d$ which minimizes $\cE_d$. To do so, we separate the terms in $\cE_d$ which depend on $q_d$ from those which do not depend on $q_d$. This depends on the particular form of the weights.

From \eqref{equ:hybridWeight} and \eqref{eq:weight2} we obtain hybrid weights
\[
  \widetilde\gamma_\setv =
 \sum_{\bsnu_{\setu(\setv)} \in \{1:\alpha\}^{|\setu(\setv)|}} 
 \bsnu_{\setu(\setv)\cap E} ! \,
 |\bsnu_{\setu(\setv)\cap {E^c}}|! 
 \prod_{j \in \setu(\setv)} \gamma_j(\nu_j)\,,
 \quad\mbox{with}\quad
 \gamma_j(\nu_j) :=
 C'_{\alpha,b}\, b^{\alpha(\alpha-1)/2}\, 2^{\delta(\nu_j,\alpha)}\beta_j^{\nu_j}.
\]
Substituting this into \eqref{eq:E-CBC} yields
\begin{align*}
  \cE_d(\bsq)
  &\,=\,
  \frac{1}{b^m} \sum_{n=0}^{b^m-1} \sum_{\emptyset \neq \setv \subseteq \{1:d\}}
  \sum_{\bsnu_{\setu(\setv)} \in \{1:\alpha\}^{|\setu(\setv)|}} 
 \bsnu_{\setu(\setv)\cap E}! \, |\bsnu_{\setu(\setv)\cap {E^c}}|!
  \bigg(\prod_{j \in \setu(\setv)} \gamma_j(\nu_j)\bigg)
  \bigg(\prod_{j \in \setv} \omega(y_j^{(n)})\bigg)\;.
\end{align*}
Every block of $\alpha$ components in the generating vector
$\bsq$ yields one component for the interlaced polynomial lattice rule.
In order to keep track of the block and position within each block, we replace the index $d$ by a double index $(s,t)$ such
that $s$ is the index for the block and $t$ is the index within the block,
that is, we set
\[
  s \,=\, \lceil d/\alpha\rceil
  \quad\mbox{and}\quad t \,=\, (d-1)\bmod \alpha + 1
  \quad\mbox{such that}\quad d \,=\, \alpha(s-1)+t\;.
\]
We now reorder the sums in $\cE_d(\bsq)$ according to $\bsnu =
(\nu_1,\ldots, \nu_s) \in \{0:\alpha\}^s$ and $\setv \subseteq \{1:d\}$ so
that the set $\setu(\setv)$ 
consists of the indices $j$ for which
$\nu_j > 0$.
This yields
\begin{align} \label{eq:Est}
  &\cE_{s,t}(\bsq)
  \,=\,
  \frac{1}{b^m} \sum_{n=0}^{b^m-1}
  \sum_{\satop{\bsnu \in \{0:\alpha\}^s}{|\bsnu|\ne 0}}
  \sum_{\satop{\setv \subseteq \{1:d\}\text{ s.t.}}
   {\setu(\setv) = \{1\le j\le s\,:\,\nu_j>0\}}} \bsnu_E! \, |\bsnu_{E^c}|!
  \bigg(\prod_{j \in \setu(\setv)} \gamma_j(\nu_j)\bigg)
      \bigg(\prod_{j \in \setv} \omega(y_j^{(n)})\bigg) \nonumber\\
  &\,=\,
  \frac{1}{b^m} \sum_{n=0}^{b^m-1} 
 (S_1(n,s,t) + S_2(n,s,t) + S_3(n,s,t)), 
\end{align}
where
\begin{align}
S_1(n,s,t) &:=     \sum_{\satop{\bsnu_{E_s} \in \{0:\alpha\}^{|E_s|}}{|\bsnu_{E_s}|\ne 0}}
  \sum_{\satop{\setv_1 \subseteq \{1:\min(d,\alpha J)\}\text{ s.t.}}
   {\setu(\setv_1) = \{1\le j\le \min(s,J)\,:\,\nu_j>0\}}} \bsnu_{E_s}! 
  \bigg(\prod_{j \in \setu(\setv_1)} \gamma_j(\nu_j)\bigg)
      \bigg(\prod_{j \in \setv_1} \omega(y_j^{(n)})\bigg), \label{def:S1}\\
S_2(n,s,t) &:= \sum_{\satop{\bsnu_{E^c_s} \in \{0:\alpha\}^{|E^c_s|}}{|\bsnu_{E^c_s}|\ne 0}}
  \sum_{\satop{\setv_2 \subseteq \{1+\min(d,\alpha J):d\}\text{ s.t.}}
   {\setu(\setv_2) = \{\min(s,J) < j\le s\,:\,\nu_j>0\}}} |\bsnu_{E^c_s}|! 
  \bigg(\prod_{j \in \setu(\setv_2)} \gamma_j(\nu_j)\bigg)
      \bigg(\prod_{j \in \setv_2} \omega(y_j^{(n)})\bigg),\label{def:S2}\\
S_3(n,s,t) &= S_1(n,s,t) \cdot S_2(n,s,t) \nonumber,
\end{align}
with $E_s := E \cap \{1:s\}$, $E^c_s := {E^c} \cap \{1:s\}$,  
$\nu_{E_s} = (\nu_j)_{{j\in E_s}}$ and $\nu_{E^c_s} = (\nu_j)_{j \in E^c_s}$.
For $s \le J$ we set $S_2(n,s,t) = 0$ and for $J = 0$ we set $S_1(n,s,t) = 0$.

We note that $S_1(n)$ has a product weight structure while $S_2(n)$ has
an SPOD weight structure. 
If $d > \alpha J$ then $S_1(n)$ is fixed and we need to compute $S_2(n)$ only.

For $d  =\alpha(s-1) + t\le \alpha J$ we have
\begin{align*}
  \cE_{s,t}(\bsq)
  &\,=\,
  \frac{1}{b^m} \sum_{n=0}^{b^m-1} 
  \sum_{ \satop{\bsnu \in \{0:\alpha\}^s}{|\bsnu|\ne 0} }
  \sum_{\satop{\emptyset \neq \setv_1 \subseteq \{1:\alpha(s-1)+t\} \text{ s.t.}}
   { \setu(\setv_1) = \{1\le j\le \min(s,J)\,:\,\nu_j>0\} } }
  \bigg(\prod_{j \in \setu(\setv_1)} \nu_j!\gamma_j(\nu_j)\bigg)
  \bigg(\prod_{j \in \setv_1} \omega(y_j^{(n)})\bigg) \\
  &\,=\,
  \frac{1}{b^m} \sum_{n=0}^{b^m-1} \sum_{\emptyset \neq \setu \subseteq \{1:s\}}
  \bigg(\prod_{j \in \setu} \sum_{\nu_j=1}^\alpha\nu_j!\gamma_j(\nu_j)\bigg)
   \sum_{\satop{\setv_1 \subseteq \{1:\alpha(s-1)+t\}}{\setu(\setv_1)=\setu}}
  \bigg(\prod_{j \in \setv_1} \omega(y_j^{(n)})\bigg)\;.
\end{align*}
Replacing $d$ by the double index $(s,t)$ as before, we obtain for
$t=\alpha$ that
\begin{align*}
  \cE_{s,\alpha}(\bsq)
  &\,=\,
  \frac{1}{b^m} \sum_{n=0}^{b^m-1} \underbrace{\prod_{j=1}^s
  \bigg[1 + \sum_{\nu_j=1}^\alpha 
     \nu_j!\gamma_j(\nu_j) \bigg(\prod_{i=1}^\alpha (1+ \omega(y_{j,i}^{(n)}) )- 1\bigg)\bigg] }_{=:\,Y_s(n)} - 1\;,
\end{align*}
where we defined the quantity $Y_s(n)$, with $Y_0(n) :=1$. For $t<\alpha$
we have
\begin{align*}
  \cE_{s,t}(\bsq)
  &\,=\,
  \frac{1}{b^m} \sum_{n=0}^{b^m-1}
  \bigg[1 + \sum_{\nu_s=1}^\alpha
\nu_s!\gamma_s(\nu_s) \bigg(\underbrace{\prod_{i=1}^t (1+ \omega(y_{s,i}^{(n)}))}_{=:\, V_{s,t}(n)}-1\bigg)\bigg]
  Y_{s-1}(n)
  - 1\;,
\end{align*}
where 
$V_{s,t}(n)$ is defined above. The part of
$\cE_{s,t}(\bsq)$ that is affected by $q_{s,t}$ is
\[
 \sum_{n=1}^{b^m-1}
 \omega(y_{s,t}^{(n)})\, V_{s,t-1}(n)\, Y_{s-1}(n)\;.
\]

In order to compute this quantity for every $q_{s,t}\in \Gc_{b,m}$ we need to perform the
matrix-vector multiplication using the matrix
\begin{equation*}
  \bsOmega \,:=\,
  \left[\omega\left(\upsilon_m\left(\frac{n(x) q(x)}{P(x)} \right)\right)
  \right]_{\satop{1\le n\le b^m-1}{q\in \Gc_{b,m}}}
\end{equation*}
and the vector $[V_{s,t-1}(n)\,Y_{s-1}(n)]_{1\le n\le b^m-1}$. A permutation can be applied to $\bsOmega$ using the so-called Rader transform (see, e.g., \cite{NC06a})) such that the fast Fourier transform can be used to carry out the matrix-vector multiplication. As shown in \cite{NC06a}, this reduces the cost of the matrix-vector multiplication to $\calO(M\,\log M) = \calO(N\,\log N)$ operations, where $M=b^m-1$ and $N = b^m$. 

Once $q_{s,t}$ has been computed for a given dimension, one has to update
the products $V_{s,t}(n)$. This can be done in $\calO(N)$ operations. After an entire block of $\alpha$ dimensions has been computed, the products $Y_{s}(n)$ need to be updated, which can be done in $\calO(N)$ operations. The total
computational cost is then $\calO(\alpha\, s\, N\,\log N)$ operations,
with a memory requirement of $\calO(N)$. 

When $d>\alpha J$  we have 
\[
 \cE_{s,t}(\bsq)
 \,=\,
 \frac{1}{b^m} \sum_{n=0}^{b^m-1}
 S_1(n,J,\alpha) + S_2(n,s,t)\cdot(1 + S_1(n,J,\alpha) ),
\]
where $S_1(n,J,\alpha)=Y_J(n)-1$. Thus $S_1(n,J,\alpha)$ has been
computed in the first part of the algorithm and is therefore now fixed.
When the final block is complete and therefore $t = \alpha$, we have
\begin{align} \label{eq:U}
 & S_2(n,s,\alpha) = 
 \sum_{\ell=1}^{\alpha (s-J)}
 \underbrace{
 \ell!
 \sum_{\satop{\bsnu \in \{0:\alpha\}^{s-J} }{|\bsnu|=\ell}}
 \prod_{\satop{j=J+1}{\nu_j>0}}^s \bigg[ \gamma_j(\nu_j)
 \bigg(\prod_{i = 1}^{\alpha} (1+\omega(y_{j,i}^{(n)})) -1 \bigg) \bigg]
 }_{=:\, U_{s,\ell}(n)}
\;,
\end{align}
where $\bsnu \in \{0:\alpha\}^{s-J}$ is given by
$\bsnu = (\nu_j)_{j\in \{J+1,J+2,\ldots,s\}}$ and where
we defined the quantity $U_{s,\ell}(n)$, with $U_{J,\ell}(n):=1$,
$U_{s,0}(n):=0$, and $U_{s,\ell}(n) :=0$ for $\ell>\alpha (s-J)$.
When the final block is incomplete, that is, $t < \alpha$, by
separating out the case $\nu_s = 0$ in \eqref{eq:Est}, we get
\begin{align*}
 & S_2(n,s,t)
 \,=\,
 \sum_{\ell=1}^{\alpha (s-1-J)}
 \ell!
 \sum_{\satop{\bsnu \in \{0:\alpha\}^{s-1-J} }{|\bsnu|=\ell}}
 \prod_{\satop{j=J+1}{\nu_j>0}}^{s-1} \bigg[ \gamma_j(\nu_j)
 \bigg(\prod_{i = 1}^{\alpha} (1+\omega(y_{j,i}^{(n)})) -1 \bigg) \bigg]
 \nonumber\\
 &\qquad\qquad\quad +
 \sum_{\ell=1}^{\alpha (s-J)}
 \sum_{\nu_s=1}^{\min(\alpha,\ell)}
 \ell!
 \sum_{\satop{\bsnu \in \{0:\alpha\}^{s-1-J} }{|\bsnu|=\ell-\nu_s}} \Bigg(
 \prod_{\satop{j=J+1}{\nu_j>0}}^{s-1} \bigg[ \gamma_j(\nu_j)
 \bigg(\prod_{i = 1}^\alpha (1+\omega(y_{j,i}^{(n)})) -1\bigg) \bigg] \nonumber\\
 &\qquad\qquad\qquad\qquad\qquad\qquad\qquad\qquad\qquad\qquad\times
 \gamma_s(\nu_s)
 \bigg(\prod_{i = 1}^t (1+\omega(y_{s,i}^{(n)})) -1\bigg) \Bigg)\;, \nonumber
\end{align*}
and thus
\begin{align}  \label{eq:VWX}
 &S_2(n,s,t)
 \,=\, S_2(n,s-1,\alpha) \\
 &\qquad + 
 \bigg(\underbrace{\prod_{i = 1}^t (1+\omega(y_{s,i}^{(n)}))}_{=:\,V_{s,t}(n)}-1 \bigg)
 \bigg(
 \underbrace{
 \sum_{\ell=1}^{\alpha (s-J)}
 \underbrace{
 \sum_{\nu_s=1}^{\min(\alpha,\ell)} \gamma_s(\nu_s)\frac{\ell!}{(\ell-\nu_s)!}\, U_{s-1,\ell-\nu_s}(n)
 }_{=:\, X_{s,\ell}(n)}
 }_{=:\, W_s(n)}
 \bigg)\;, \nonumber
\end{align}
where we defined 
$V_{s,t}(n)$, $W_{s}(n)$, and $X_{s,\ell}(n)$ as indicated, with
$V_{s,0}(n):=1$.

Since the polynomial $q_{s,t}$ only appears in the final factor of the
products $V_{s,t}(n)$, the only part of $\cE_{s,t}(\bsq)$ that is
affected by $q_{s,t}$ is
\[
 \sum_{n=1}^{b^m-1}
 \omega(y_{s,t}^{(n)})\, V_{s,t-1}(n)\, W_s(n)\,(1+S_1(n,J,\alpha))\;.
\]
Computing this quantity for every $q_{s,t}\in \Gc_{b,m}$ requires the
matrix-vector multiplication with the matrix $\bsOmega$
%
and the vector $[V_{s,t-1}(n)\,W_s(n)\,(1+S_1(n,J,\alpha))]_{1\le n\le b^m-1}$. 
Again, one can apply a permutation to the matrix $\bsOmega$ such that the fast 
Fourier transform can be used \cite{NC06a}. The cost is 
then $\calO(M\,\log M) = \calO(N\,\log N)$ operations, where $M=b^m-1$ and $N = b^m$. 

Once $q_{s,t}$ is chosen for dimension $\alpha (s-1) + t$, we update the products $V_{s,t}(n)$ using
\[
  V_{s,t}(n) \,=\, (1+\omega(y_{s,t}^{(n)}))\, V_{s,t-1}(n)\;.
\]
This requires $\calO(N)$ operations. After completing an entire
block of $\alpha$ dimensions, also the values $U_{s,\ell}(n)$ need to be updated. This can be done using the equation
\begin{align*}
 U_{s,\ell}(n)
 &\,=\,
 \ell!
 \sum_{\satop{\bsnu \in \{0:\alpha\}^{s-J-1} }{|\bsnu|=\ell}}
 \prod_{\satop{j=J+1}{\nu_j>0}}^{s-1} \bigg[ \gamma_j(\nu_j)
 \bigg(\prod_{i = 1}^{\alpha} (1+\omega(y_{j,i}^{(n)})) -1\bigg) \bigg] \\
 &\qquad
 +
 \ell! 
 \sum_{\nu_s=1}^{\min(\alpha,\ell)} \!\!
 \sum_{\satop{\bsnu \in \{0:\alpha\}^{s-J-1} }{|\bsnu|=\ell-\nu_s}} \!\!\!
 \Bigg(
 \prod_{\satop{j=J+1}{\nu_j>0}}^{s-1} \bigg[ \gamma_j(\nu_j)
 \bigg(\prod_{i = 1}^{\alpha} (1+\omega(y_{j,i}^{(n)})) -1\bigg) \bigg] \\
 &\qquad\qquad\qquad\qquad\qquad\qquad\times
 \gamma_s(\nu_s)
 \bigg(\prod_{i = 1}^{\alpha} (1+\omega(y_{s,i}^{(n)}))-1 \bigg) \Bigg)\\
&\,=\, U_{s-1,\ell}(n) + (V_{s,\alpha}(n) -1)\,X_{s,\ell}(n)\;.
\end{align*}
Since the quantities $V_{s,\alpha}(n)$ and $X_{s,\ell}(n)$ 
can be pre-computed and stored, this update requires 
$\calO(\alpha\, (s-J)_+  N)$ operations,  where $(x)_+ = \max\{0, x\}$. 
In the next step, the products $V_{s+1,0}(n)$ need to be initialized by $1$
with $\calO(N)$ operations, and the quantities $W_{s+1}(n)$ and
$X_{s+1,\ell}(n)$ need to be computed, which can be done in $\calO(\alpha^2 (s-J)_+ N)$ operations. The algorithm then continuous the search in the new block.

We need to store the quantities $U_{s,\ell}(n)$, $V_{s,t}(n)$, $W_{s}(n)$,
and $X_{s,\ell}(n)$, which 
can be overwritten as we increase $s$ and $t$. Hence, the total memory
requirement is $\calO(\alpha\, s\, N)$.

The 
{\em total computational cost for the CBC construction up to {dimension} $\alpha s$}
is therefore bounded by
$$
\begin{array}{rl}
  \calO\left(\alpha\, \min\{s, J\} \, N \, \log N + \alpha^2 (s-J)_+  \, N\,\log N \right) 
& \mbox{ search cost, plus}  
  \\ 
  \calO(\alpha^2 (s-J)_+^2 N) & \mbox{update cost, plus} 
  \\ 
  \calO(N + \alpha (s-J)_+ N) & \mbox{memory cost}\;.
\end{array}
$$ 
Hence, 
for large values of $J$ (as may occur in practice,
cf. Remark \ref{remk:ChoicS}), and for higher orders $\alpha$
the product structure of the QMC weights up 
to dimension $J$, implied by \eqref{eq:HybBd},
imply quantitative advantages in the CBC construction.

We summarize the algorithm in Pseudocode~1 below;
there, $.*$ means element-wise multiplication. 
Note that $\bsU(\ell)$ for $\ell=0,\ldots, \alpha (s_{\max}-J)$, 
and $\bsV$, $\bsW$, $\bsX(\ell)$ for $\ell=1,\ldots, \alpha (s_{\max}-J)$, 
and ${\boldsymbol \cE}$ are all vectors of length $N-1$, 
while $\bsOmega^{\rm perm}$ denotes the permuted version of the
matrix $\bsOmega$. 
The vector ${\boldsymbol \cE}$ stores the values of $\cE_d$.

\begin{algorithm}[ht] 
\caption{(Fast CBC implementation for hybrid weights)}
\small
\begin{algorithmic}
  \State $\bsY := \bsone$
  \For{\bf $s$ from $1$ to $\min(J,s_{\max})$}
  \State $\bsV := \bsone$
  \For{\bf $t$ from $1$ to $\alpha$}
  \State
  ${\boldsymbol \cE} := \bsOmega^{{\rm perm}}\,(\bsV .\!* \bsY)$
  \Comment{compute -- { use FFT}}
  \State
  $q_{s,t} := {\rm argmin}_{q\in \Gc_{b,m}} \cE(q)$
  \Comment{select -- { pick the correct index}}
  \State
  $\bsV := \big(\bsone + \bsOmega^{{\rm perm}}(q_{s,t},:)\big) \,.\!*\,\bsV$
  \Comment{update products}
  \EndFor
  \State $\bsY := (\bsone + \sum_{\nu_j=1}^\alpha\nu_s!\gamma_s (\bsV - \bsone)) \,.\!*\, \bsY$
  \Comment{update products}
  \EndFor

\State
$\bsS_1 := \bsY - \bsone$

\If{$s_{\max} \le J$} \Return  \EndIf 

\State $\bsU(0) := \bsone$
  \State $\bsU(1:\alpha\, (s_{\max}-J)) := \bszero$
  \For{\bf $s$ from $J+1$ to $s_{\max}$}
  \State $\bsV := \bsone$
  \Comment{initialize products and sums}
  \State $\bsW := \bszero$
  \For{\bf $\ell$ from $1$ to $\alpha (s-J)$}
  \State $\bsX(\ell) := \bszero$
  \For{\bf $\nu$ from $1$ to $\min(\alpha,\ell)$}
  \State $\bsX(\ell) := \bsX(\ell) + \gamma_s(\nu) \displaystyle\frac{\ell!}{(\ell-\nu)!}\, \bsU(\ell - \nu)$
  \EndFor
  \State $\bsW := \bsW + \bsX(\ell)$
  \EndFor
  %
 \For{\bf $t$ from $1$ to $\alpha$}
  \State
  ${\boldsymbol \cE} := \bsOmega^{{\rm perm}}\,(\bsS_1 + (\bsone + \bsS_1).\!*\bsV .\!* \bsW)$
  \Comment{compute -- { use FFT}}
  \State
  $q_{s,t} := {\rm argmin}_{q\in \Gc_{b,m}} \cE(q)$
  \Comment{select -- { pick the correct index}}
  \State
  $\bsV := \big(\bsone + \bsOmega^{{\rm perm}}(q_{s,t},:)\big) \,.\!*\,\bsV$
  \Comment{update products}
  \EndFor
  %
  \For{\bf $\ell$ from $1$ to $\alpha (s-J)$}
  \Comment{update sums}
  \State $\bsU(\ell) := \bsU(\ell) + (\bsV - \bsone) \,.\!*\, \bsX(\ell)$
  \EndFor
  \EndFor
\end{algorithmic}
\end{algorithm}
\section{Conclusion}
\label{sec:Concl}
We have analyzed the convergence of a class of higher order
Quasi Monte-Carlo (HOQMC) quadrature methods for the approximate evaluation
of response-statistics of a class of nonlinear operator equations
subject to distributed uncertainty, corresponding (via an unconditional Schauder
basis) to infinite-dimensional, iterated integrals.
We showed that for operators with analytic dependence on the
uncertain input, the HOQMC quadratures achieve
convergence rates which are independent of the parameter dimension
and which are, in a sense, best possible for a given sparsity
measure of the parameter dependence.
The main result of the present paper,
Theorem \ref{thm:DsiboundC}, 
is of independent interest and has applications
beyond the QMC quadrature error analysis for 
parametric operator equations developed in the present
paper:
countably-parametric integrand 
functions with $(\bsb,p,\eps)$-holomorphic dependence on 
the components $y_j$ of the parameter vector $\bsy$
admit high order quasi Monte-Carlo quadratures with \emph{dimension-independent}
convergence rates of order $1/p$. As the proof
of Theorem \ref{thm:DsiboundC} involved analytic continuation,
analogous results hold also in other contexts, such as 
Bayesian inverse problems which will be considered in 
\cite{DLGCSInv}.
We point out that the high order quasi Monte-Carlo error bounds
in Proposition \ref{prop:main1} require only finite differentiability
of the integrand function with respect to the integration parameters;
therefore the present convergence analysis also applies
to classes of nonanalytic integrand functions $G(\bsy)$, 
even with finite smoothness,
as long as quantitative bounds on its derivatives are available 
that are explicit with respect to the dimension $s$ of the domain
of integration; we refer to \cite{HoSc12Wave} for an example.
Theorem \ref{thm:DsiboundC}
allows us to control derivatives of the integrand functions
of arbitrary order, with constants which are explicit in the
derivative order and independent of the dimension of the 
parameter space.
Applications of the presently proposed,
combined high order quasi Monte-Carlo quadrature with 
Petrov-Galerkin discretizations of the forward problems
to Bayesian inversion in uncertainty quantification 
will be considered in \cite{DLGCSInv}.
There, the posterior densities admit
an infinite-dimensional, parametric deterministic
representation which, as we show, ``inherits''
analyticity from the forward map 
(cp. also \cite{ScSt11,SS12,SS13} and the references there). 
Details on the extension of the present analysis to this problem
class will be available in \cite{DLGCSInv}, with Theorem \ref{thm:DsiboundC}
of the present paper taking again a key role.
Numerical tests confirming the results of the present paper and implementation details are provided in \cite{GSc14}.

In the present paper, we have confined the
analysis to the so-called single-level version of the HOQMC-PG
discretization, 
and assumed minimal regularity $G(\cdot)\in \bcX'$.
Based on the present results, {\em multilevel discretizations} 
can be designed which are more complicated but
which are expected to exhibit, in certain cases, 
superior performance (we refer to \cite{DKLS15}
for the analysis of a 
higher order, multilevel QMC-PG algorithm 
    in the particular case of affine-parametric, linear operators).
The analysis of such multilevel algorithms in the present 
general context, will likewise be presented elsewhere.

\end{document}